\documentclass{amsart}
\usepackage{amssymb,amsmath,amsthm}
\usepackage{graphicx}
\title[Bihermitian Geometry and the Holomorphic Sections of Twistor Space]{Bihermitian Geometry and the Holomorphic Sections of Twistor Space}
\author{Steven Gindi}
\newtheorem{thm}{Theorem}[section]
\newtheorem{lemma}[thm]{Lemma}
\newtheorem{prop}[thm]{Proposition}
\newtheorem{cor}[thm]{Corollary}
\newtheorem{rmk}[thm]{Remark}

\newtheorem{nota}[thm]{Notation}
\theoremstyle{definition} \newtheorem{example}[thm]{Example}
\theoremstyle{definition}  \newtheorem{defi}[thm]{Definition}
\numberwithin{equation}{section}
\begin{document}
\begin{large}
\begin{abstract}
We use our recently introduced holomorphic twistor spaces to derive results about the complex geometries of their base manifolds. In particular, we use these twistor spaces to develop a new approach to studying generalized Kahler manifolds. This leads to insights into their real and holomorphic Poisson structures.
\end{abstract}
\maketitle
\tableofcontents

\newpage
\section{Introduction}
Recently, in \cite{Gindi1} we introduced integrable complex structures on twistor spaces fibered over complex manifolds, equipped with certain geometrical data. The resulting holomorphic spaces were shown to arise naturally in different contexts such as when the base manifold is bihermitian, also known as generalized Kahler (\cite{Rocek1, Apost1, Gualt1}). In this paper, we demonstrate how to use these twistor spaces to derive results about the complex geometries of the base manifold. In particular, we develop a new way of thinking about bihermitian manifolds that leads to insights into their real and holomorphic Poisson structures. 

Our first application of holomorphic twistor spaces is given in Section \ref{SecHSMAIN}. There we use their holomorphic sections to decompose the base manifold into different types of holomorphic subvarieties, denoted by $M^{\delta}$ (Theorem \ref{ThmPHS}). The main idea behind this construction is that the holomorphic sections of twistor space not only induce holomorphic bundles over the base manifold but different types of holomorphic bundle maps as well. Some of the $M^{\delta}$ then correspond to the degeneracy loci of these maps while others refine their structure. 

We then establish in Section \ref{SecATVP} a twistor point of view of the $M^{\delta}$ by realizing them as intersections of different complex submanifolds and holomorphic subvarieties in twistor space. This allows us to develop tools to study the $M^{\delta}$ inside  this space and leads us to derive a number of results about them. Our first result is given in Section \ref{SecBINT} where we establish lower bounds on their dimensions.  Secondly we determine necessary conditions for there to exist curves in the base manifold that lie in certain $M^{\delta}$ (Propositions \ref{PropMPC} and \ref{PropMPC2}). As described in Section \ref{SecMAC}, these conditions lead to \textit{upper bounds} on the dimensions of the subvarieties. 

To demonstrate the importance of these results, we will now describe 
two of our major classes of examples of a holomorphic twistor space equipped with holomorphic sections---when the base manifold is a bihermitian manifold and a  holomorphic twistor space. As part of the second example, we will provide more details as to how we used twistor spaces to derive the above results. 

\subsection{Bihermitian and Generalized Kahler Manifolds}
  A bihermitian manifold is a Riemannian manifold equipped with a pair of complex structures that satisfy certain relations (Section \ref{SecBG}). These manifolds were first introduced by physicists in \cite{Rocek1}, as the target spaces of  supersymmetric sigma models, and were later found to be equivalent to (twisted) generalized Kahler manifolds \cite{Gualt1,Hitchin1} (see also \cite{Apost1}). Consequently, there are several approaches in the literature that are used to study these manifolds; and in this paper, we introduce yet another---we study them via their twistor spaces.
  
  Indeed, one of our major results of \cite{Gindi1} is that the twistor space of a bihermitian manifold admits two integrable complex structures. In Section \ref{SecBM} of the present paper we further demonstrate that these holomorphic twistor spaces admit natural holomorphic sections. By then applying Theorem \ref{ThmPHS} to this case, we decompose the  bihermitian manifold in Section \ref{SecBM} into holomorphic subvarieties that are new to the literature. 
  
  The importance of these subvarieties lies in their connection to known Poisson structures on the  manifold: Some of the subvarieties are the degeneracy loci of a holomorphic Poisson structure while the ones that are new to the literature are surprisingly the loci of  \textit{real} Poisson structures. At the same time, there are others that  refine the structure of both of these loci. As a consequence, we can now study the Poisson structures on a bihermitian manifold by using the new tools from twistor spaces that were described above. 

For instance, by applying the general bounds of Theorem \ref{ThmBOUNDS}, we derive in Section \ref{secBMBPG} existence results about the subvarieties in bihermitian manifolds. In particular, we demonstrate that there are classes of bihermitian structures on $\mathbb{CP}^{3}$ that cannot admit certain $M^{\delta}$. As these subvarieties refine the degeneracy loci of the corresponding holomorphic Poisson structures on $\mathbb{CP}^{3}$, our results provide new information about the structure of these loci.

\subsection{Stratifications of Twistor Spaces}
Our second major class of examples of a twistor space that is equipped with holomorphic sections is when the base manifold is itself a holomorphic twistor space (Section \ref{SecSTS1}). In this case, we not only produce different stratifications of twistor spaces, whose strata are complex submanifolds  and holomorphic subvarieties, but also use these structures to derive the results about the general $M^{\delta}$ of Section \ref{SecGTM}.

As we show in Section \ref{SecSTS1}, some of the complex submanifolds that we produce in twistor space can be viewed as Schubert cells in a certain  Grassmannian space. By using this correspondence and defining special charts for the twistor space, we determine the dimensions  of these submanifolds (as well as the dimensions of the other subvarieties) and describe their tangent bundles.  

These properties are in fact important in our derivation of the results about the $M^{\delta}$ given in Section \ref{SecGTM}. The way that we derive them is to first holomorphically embed the base manifold into its twistor space and then, as mentioned above, to realize the $M^{\delta}$ as intersections of the different complex submanifolds and holomorphic subvarieties (Section \ref{SecATVP}).  One advantage of this point of view is that the codimension of the intersection of any two holomorphic subvarieties is always bounded from above by the sum of their codimensions.  Being that we have already determined the dimensions of the subvarieties in twistor space in Section \ref{SecTB23}, we arrive at the bounds on the $M^{\delta}$ given in Section \ref{SecBINT}. Moreover, we also apply  the description of the tangent bundles given in Section \ref{SecTB23} to derive the necessary conditions for there to exist curves that lie in certain $M^{\delta}$ as specified in Propositions \ref{PropMPC} and \ref{PropMPC2}. 

We will now begin by reviewing the integrable complex structures on twistor spaces that we introduced in \cite{Gindi1}.

\section{Complex Structures on Twistor Spaces}
Let $E \longrightarrow (M,I)$ be a rank $2n$ real vector bundle that is fibered over a complex manifold, with complex structure $I$, and let $\mathcal{C}(E)= \{J \in EndE | \ J^{2}=-1\}$ be its twistor space. In the case when $g$ is a positive definite, fiberwise metric on $E$ we will also consider the twistor space $\mathcal{T}(E,g)= \{J \in \mathcal{C}(E) | \ g(J\cdot,J\cdot)=g(\cdot,\cdot)\}$.

\begin{nota} At times we will denote $\mathcal{C}(E)$ by $\mathcal{C}$ and $\mathcal{T}(E,g)$ by $\mathcal{T}$.
\end{nota}

Letting $\nabla$ be a connection on $E$, we will now define the almost complex structure $\mathcal{J}^{(\nabla,I)}$ on the total space of $\pi: \mathcal{C}(E) \longrightarrow M$ (where $\pi$ is the natural projection map): 
\begin{defi} 
\begin{enumerate}
 \item[]
 \item[1)]  Use $\nabla$ to split $T\mathcal{C}$ into $V\mathcal{C} \oplus H^{\nabla}\mathcal{C}$, the direct sum of  vertical and horizontal distributions (see \cite{Gindi1} for more details). \\
 \item[2)] Define \begin{equation*}
 \mathcal{J}^{(\nabla,I)}= \mathcal{J}^{V} \oplus \pi^{*}I,
 \end{equation*}
 where $\mathcal{J}^{V}$ is the standard fiberwise complex structure on $V\mathcal{C}$ and where we have identified $H^{\nabla}\mathcal{C}$ with $\pi^{*}TM$. 
 \end{enumerate}
 \end{defi}
 Letting $R^{\nabla}$ be the curvature of $\nabla$, in \cite{Gindi1} we proved
 
 \begin{thm}
 \label{thmINT}
 If $R^{\nabla}$ is of type (1,1), i.e. $R^{\nabla}(I\cdot,I\cdot)= R^{\nabla}(\cdot,\cdot)$, then $\mathcal{J}^{(\nabla,I)}$ is an integrable complex structure on $\mathcal{C}$.
 \end{thm}
 
 Assuming that $R^{\nabla}$ is (1,1), we have 

\begin{prop} 
\label{PropHS}
$\pi: (\mathcal{C}, \mathcal{J}^{(\nabla,I)}) \longrightarrow (M,I)$ is a holomorphic submersion.
\end{prop}

In the case when $g$ is a fiberwise metric on $E$ and $\nabla$ is a metric connection, we have: 
\begin{prop}  $\mathcal{T}$ is a complex submanifold of $(\mathcal{C},\mathcal{J}^{(\nabla,I)})$. 
\end{prop}
 
\section{Some Examples}

In \cite{Gindi1}, we described various examples of bundles that admit connections with (1,1) curvature and the resulting complex structures on the twistor spaces. In this section, we recall how to define complex structures on the twistor space associated to any Hermitian manifold equipped with a $\overline{\partial}$ closed (2,1) form and, in particular, to any SKT or bihermitian manifold. For this, we will be using a general correspondence between connections on $E \longrightarrow (M,I)$ with (1,1) curvature and $\overline{\partial}-$operators on $E_{\mathbb{C}}:=E \otimes_{\mathbb{R}}\mathbb{C}$. To describe it, let us $\mathbb{C}$-linearly extend a connection $\nabla$ on $E$ to a complex connection on $E_{\mathbb{C}}$ and denote the corresponding (0,1) connection by $\nabla^{0,1}$; we then have: 
\begin{lemma}
\label{LemCOR}
$R^{\nabla}$ is (1,1) if and only if $(\nabla^{0,1})^{2}=0.$ Moreover, given a $\overline{\partial}-$operator  $\overline{\partial}$ on $E_{\mathbb{C}}$, there is a unique connection $\nabla$ on $E $ such that $\nabla^{0,1}= \overline{\partial}$. 
\end{lemma}

\subsection{Three Forms}
\label{SecTF}
Now let $(M,g,I)$ be a Hermitian manifold equipped with a real three form $H= \overline{H^{2,1}} +H^{2,1}$ of type (1,2) + (2,1), such that $\overline{\partial}H^{2,1}=0$. We will use $H$ to define a connection on $TM$ with (1,1) curvature by first defining a $\overline{\partial}$-operator on $TM_{\mathbb{C}}$ and then by using the correspondence given in Lemma \ref{LemCOR}. If we denote the Chern connection on $TM$ by $\nabla^{Ch}$ then the $\overline{\partial}$-operator on $TM_{\mathbb{C}}=T^{1,0} \oplus T^{0,1}$ that we will consider is $\nabla^{Ch(0,1)} + g^{-1}H^{2,1}$. (Here, we are viewing $g^{-1}H^{2,1}$ as a section of $T^{*0,1} \otimes \mathfrak{so}(TM_{\mathbb{C}})$ by setting $g^{-1}H^{2,1}_{v}w=g^{-1}H^{2,1}(v,w,\cdot)$, for $v \in T^{0,1}$ and $w \in TM_{\mathbb{C}}$.) 
As the corresponding real connection is $\nabla^{Ch}+ \frac{1}{2}I[g^{-1}H,I]$, we have  
\begin{prop}
\label{PropTF} 
\begin{align*}
1) & \  \nabla= \nabla^{Ch}+ \frac{1}{2}I[g^{-1}H,I] \text{ is a metric connection on } TM \text{ with } \\ & \text { (1,1) curvature.} \\
2) &  \ \mathcal{J}^{(\nabla,I)} \text{ is a complex structure on } \mathcal{C} \text{ and } \mathcal{T}.
\end{align*}
 \end{prop}

\subsubsection{SKT Manifolds} 
\label{SecSKT}
A Hermitian manifold, $(M,g,I)$, is by definition SKT (strong Kahler with torsion) if the three form $H=-d^{c}w=i(\partial -\overline{\partial})w$ satisfies $dH=0$, where $w(\cdot,\cdot)=g(I\cdot,\cdot)$ \cite{Skt2,Skt1}. As this condition is equivalent to $\overline{\partial}H^{2,1}=0$, by Proposition \ref{PropTF} $ \nabla^{Ch}- \frac{1}{2}I[g^{-1}H,I]$ is a particular connection on $TM$ with (1,1) curvature. It can be shown to equal $\nabla^{-}:= \nabla^{Levi} -\frac{1}{2}g^{-1}H$, where $\nabla^{Levi}$ is the Levi Civita connection, and is closely related to the Bismut connection $\nabla^{+}:= \nabla^{Levi} +\frac{1}{2}g^{-1}H$ \cite{Bis1,Guad1}. We thus have

\begin{cor}
\label{CorSKTCS}
If $(M,g,I)$ is SKT then $(\mathcal{C}, \mathcal{J}^{(\nabla^{-},I)})$ is a complex manifold and $\mathcal{T}$ is a complex submanifold.
  \end{cor}
\subsubsection{Bihermitian Manifolds}
\label{SecBG}
A source of SKT manifolds is bihermitian manifolds, also known as generalized Kahler manifolds \cite{Rocek1,Apost1, Gualt1,Hitchin1}. A bihermitian manifold is by definition a Riemannian manifold $(M,g)$ that is equipped with two metric compatible complex structures $J_{+}$ and $J_{-}$ that satisfy the following conditions
\[\nabla^{+}J_{+}=0 \ \ \text{ and } \ \  \nabla^{-}J_{-}=0, \]
where $\nabla^{\pm}= \nabla^{Levi} \pm \frac{1}{2}g^{-1}H$, for a closed three form $H$.

It can be shown that $\nabla^{+}$ and $\nabla^{-}$ are the respective Bismut connections for $(g,J_{+})$ and $(g,J_{-})$. Thus an equivalent way to express the above bihermitian conditions is  
\[H=-d^{c}_{+}w_{+}=d^{c}_{-}w_{-} \ \ \text{ and } \ \ dH=0.\]
Since $dH$ is assumed to be zero, $(g,J_{+})$ and $(g,J_{-})$ are two SKT structures on $M$ and hence by Corollary \ref{CorSKTCS} we have: 
\begin{cor}
\label{CorCSBG}$\mathcal{J}^{(\nabla^{-},J_{+})}$ and $ \mathcal{J}^{(\nabla^{+},J_{-})}$ are two complex structures on $\mathcal{C}$ and $\mathcal{T}$.
\end{cor}
\section{Holomorphic Sections of Twistor Space}
\label{SecHST}
Given a bundle $E \longrightarrow (M,I)$, equipped with a connection $\nabla$ that has (1,1) curvature, we will presently give different characterizations of the holomorphic sections of $\pi: (\mathcal{C}(E), \mathcal{J}^{(\nabla,I)}) \longrightarrow M$ and show how these sections yield new connections on $E$ with (1,1) curvature. We will then use these results to decompose $M$ into holomorphic subvarieties in Section \ref{SecHSMAIN}. 

 The following gives a first characterization of the holomorphic sections of twistor space.
\begin{prop}
\label{PropHSC}
 The section $J:(M,I) \longrightarrow (\mathcal{C}, \mathcal{J}^{(\nabla,I)})$ is holomorphic if and only if $J\nabla_{v} J= \nabla_{Iv}J$, for all $v \in TM$. 
\end{prop}
\begin{proof}
Letting $P^{\nabla}:T\mathcal{C} \longrightarrow V\mathcal{C}$ be the projection operator that is based on the splitting of $T\mathcal{C}$ into $V\mathcal{C} \oplus H^{\nabla}\mathcal{C}$, let us consider the holomorphicity condition of $J$: $\mathcal{J}^{(\nabla,I)}J_{*}=J_{*}I$. If $v \in T_{x}M$, we then have: 

 1) $\mathcal{J}^{(\nabla,I)}J_{*}v= \mathcal{J}^{(\nabla,I)}(P^{\nabla}(J_{*}v) + v^{\nabla})$, where $v^{\nabla} \in H^{\nabla}_{J(x)}\mathcal{C}$ is the horizontal lift of $v \in T_{x}M$. This then equals $JP^{\nabla}(J_{*}v) + (Iv)^{\nabla}$.
 
2) $J_{*}(Iv)= P^{\nabla}(J_{*}Iv) + (Iv)^{\nabla}$.

Hence $J$ is holomorphic if and only if 
\begin{equation} 
\label{EqH}
JP^{\nabla}(J_{*}v)= P^{\nabla}(J_{*}Iv), 
\end{equation}
 for all $v\in TM$. Using \cite{Gindi1}, it is straightforward to show that $P^{\nabla}(J_{*}v)=\nabla_{v}J$. Plugging this into Equation \ref{EqH} proves Proposition \ref{PropHSC}.  
\end{proof}

If we consider the $\overline{\partial}-$operator $\nabla^{0,1}$ on $E_{\mathbb{C}}$, as described in Lemma \ref{LemCOR}, then the above holomorphicity condition is equivalent to $(\nabla^{0,1}J)E^{0,1}_{J}=0.$ This in turn is equivalent to $ J\nabla^{0,1}e=-i\nabla^{0,1}e,$ for all $e \in \Gamma(E^{0,1}_{J})$. We thus have: 
\begin{prop}
\label{PropDBAR}
The section $J:M \longrightarrow (\mathcal{C}, \mathcal{J}^{(\nabla,I)})$ is holomorphic if and only if $E^{0,1}_{J}$ is a holomorphic subbundle of $(E_{\mathbb{C}},\nabla^{0,1})$.
\end{prop}

Let us now use a holomorphic section to build other connections on $E$ with (1,1) curvature.

\begin{prop}
Let $J:M \longrightarrow (\mathcal{C}, \mathcal{J}^{(\nabla,I)})$ be a holomorphic section. $\nabla +\nabla J(a+bJ)$, where $a,b \in \mathbb{R}$, is a connection on $E$ with (1,1) curvature.
\end{prop}
\begin{proof}  We will show that $\nabla^{0,1} +\nabla^{0,1} J(a+bJ)$   is a $\overline{\partial}$-operator on $E_{\mathbb{C}}$. It is straightforward to show that this (0,1) connection is of the form $\nabla^{0,1} + A$, where $A \in \Gamma(T^{*0,1} \otimes EndE_{\mathbb{C}})$ satisfies $\nabla^{0,1}A=0$, $AE^{1,0}_{J} \subset E^{0,1}_{J}$ and $AE^{0,1}_{J}=0.$ It then follows that $\nabla^{0,1} + A$ squares to zero. 
\end{proof}
Among the above connections, there is a particular one that we wish to focus on:
\begin{prop}
\label{PropNAB'}
$\nabla':=\nabla +\frac{1}{2}(\nabla J)J$ is a connection on $E$ with (1,1) curvature and  satisfies $\nabla^{'}J=0$.
\end{prop}

We will now give some examples of holomorphic sections of twistor spaces.

\subsection{Example: Three Forms}
\label{ExTF2}
As in Section \ref{SecTF}, let $(M,g,I)$ be a Hermitian manifold that is equipped with a real three form $H= \overline{H^{2,1}} +H^{2,1}$ such that $\overline{\partial}H^{2,1}=0$. Recalling Proposition \ref{PropTF}, we have 
\begin{prop} $I: M \longrightarrow (\mathcal{T}(TM), \mathcal{J}^{(\nabla,I)})$ is a holomorphic section, where $\nabla= \nabla^{Ch} + \frac{1}{2}I[g^{-1}H,I]$.
\end{prop}
\begin{proof}
The holomorphicity condition is $I[g^{-1}H_{v},I]=[g^{-1}H_{Iv},I]$ for all $v\in TM$. This is equivalent to the condition that $H$ is (1,2) + (2,1).
\end{proof}
In this case, $\nabla'= \nabla +\frac{1}{2}(\nabla I)I$ is a familiar connection: $\nabla^{Ch}$.

As explained in Section \ref{SecSKT}, an SKT manifold falls into the above setup. Using Corollary \ref{CorSKTCS}, we thus have:
\begin{cor}
\label{CorSKTHI}
Let $(M,g,I)$ be an SKT manifold. $I:M \longrightarrow (\mathcal{T}, \mathcal{J}^{(\nabla^{-},I)})$ is holomorphic. 
\end{cor}
Consequently, the holomorphic twistor spaces of a bihermitian manifold,
as described in Section \ref{SecBG}, admit holomorphic sections. They will be considered in detail in Section \ref{SecBM}.

We will now provide another example where the base manifold is a holomorphic twistor space itself. 

\subsection{Example: Holomorphic Twistors}

\label{ExTHS}
Let $E \longrightarrow (M,I)$ be equipped with a connection $ \nabla$ that has (1,1) curvature.  Denoting the projection map from $\mathcal{C}(E)$ to $M$ by $\pi$, in this example we will be focusing on the complex manifold $(\mathcal{C}(E),\mathcal{I})$, where $\mathcal{I}=\mathcal{J}^{(\nabla,I)}$, along with its pullback  bundle $\pi^{*}E$. Since $\pi$ is holomorphic, the connection $\pi^{*}\nabla$ on $\pi^{*}E$ has (1,1) curvature so that by Theorem \ref{thmINT} the total space of $(\mathcal{C}(\pi^*E),\mathcal{J}^{(\pi^{*}\nabla, \mathcal{I})}) \longrightarrow \mathcal{C}(E)$ is a complex manifold. 
Moreover, the section $\phi$ of $\mathcal{C}(\pi^*E)$ defined by $\phi|_{K}=K$ is holomorphic: 
\begin{prop}
\label{PropTHS}
$\phi: \mathcal{C}(E) \longrightarrow (\mathcal{C}(\pi^*E),\mathcal{J}^{(\pi^{*}\nabla,\mathcal{I})}) $ is a holomorphic section.\end{prop}
\begin{proof}
It follows from Proposition \ref{PropHSC} that $\phi$ is holomorphic if and only if $\phi (\pi^{*}\nabla)_{X}\phi = \pi^{*}\nabla_{\mathcal{I}X}\phi$, for all $X \in T\mathcal{C}(E).$ Using \cite{Gindi1}, this is equivalent to $\phi P^{\nabla}(X)= P^{\nabla}(\mathcal{I}X),$ where $P^{\nabla}: T\mathcal{C}(E) \longrightarrow V\mathcal{C}(E)$ is the vertical projection operator that is induced by $\nabla$. This last expression follows directly from the definition of $\mathcal{I}=\mathcal{J}^{(\nabla,I)}$.
\end{proof}
It then follows from Proposition \ref{PropNAB'} that the connection $\pi^{*}\nabla'= \pi^{*}\nabla + \frac{1}{2} (\pi^{*}\nabla \phi)\phi$ on $\pi^{*}E \longrightarrow (\mathcal{C}(E),\mathcal{I})$ also has (1,1) curvature and satisfies $\pi^{*}\nabla' \phi=0$. Hence 
$\mathcal{J}^{(\pi^{*}\nabla',\mathcal{I})}$ is another complex structure on $\mathcal{C}(\pi^*E)$. 
 
\section{Holomorphic Subvarieties}
\label{SecHSMAIN}
We will now use holomorphic sections of $(\mathcal{C}(E),\mathcal{J}^{(\nabla,I)})$, as well as the corresponding connections with (1,1) curvature that were described above, to decompose $(M,I)$ into different types of holomorphic subvarieties.  

\subsection{The $M_{\leq s}$ and $M_{(\leq r, \pm)}$}
\label{SecTHS}
To begin, consider $E \longrightarrow (M,I)$ equipped with a connection $\nabla$ that has (1,1) curvature and suppose that $J$ and $K$ are respectively parallel and holomorphic sections of $(\mathcal{C}(E),\mathcal{J}^{(\nabla,I)})$---so that $\nabla J=0$ and $K\nabla K=\nabla_{I}K$. Our main theorem of this section is that the degeneracy loci of the real bundle maps $[J,K]$, $J+K$ and $J-K$ are holomorphic subvarieties of $M$:
\begin{thm}
\label{ThmPHS}
Let $J$ and $K$ respectively be parallel and holomorphic sections of $(\mathcal{C},\mathcal{J}^{(\nabla,I)})\longrightarrow (M,I)$. The following are holomorphic subvarieties of $M$:
\begin{align*}
&1)\ M_{\leq s}=\{x\in M | \ Rank[J,K]|_{x} \leq 2s \} \\ 
&2)\ M_{(\leq r, \pm)}=\{x\in M | \ Rank(J \pm K)|_{x} \leq 2r \}.
\end{align*}
\end{thm}
\begin{nota} 
\label{nota123}
We will similarly define $M_{s}$ and $M_{(r, \pm)}$ as above but with the appropriate $\leq$ signs replaced with $=$.
\end{nota}
The reason that the above real bundle maps yield holomorphic subvarieties is that they are in fact holomorphic when restricted to the appropriate holomorphic bundles, which we now describe. 

First consider the holomorphic bundle $E_{\mathbb{C}}$ that is equipped with the $\overline{\partial}$-operator $ \nabla^{0,1}$ (see Lemma \ref{LemCOR}). Since $J$ is parallel, $E^{1,0}_{J}$ and $E^{0,1}_{J}$ are two holomorphic subbundles of $E_{\mathbb{C}}$, and since $K$ is holomorphic, by Proposition \ref{PropDBAR}, 
$E^{0,1}_{K}$ is a third. Now the holomorphicity of $K$, as explained in Section \ref{SecHST}, can also be used to show that the connection $\nabla'=\nabla + \frac{1}{2}(\nabla K)K$ has (1,1) curvature, so that $\nabla'^{0,1}$ is another $\overline{\partial}$-operator on $E_{\mathbb{C}}$. As $\nabla'K=0$,  $E^{1,0}_{K}$ and  $E^{0,1}_{K}$ are holomorphic subbundles. (Note that since $(\nabla^{0,1}K)E^{0,1}_{K}=0$, $\nabla^{0,1}= \nabla'^{0,1}$ when acting on $E^{0,1}_{K}$.)

Given these bundles, we have  
\begin{prop}
\label{PropBUNM}
The following are holomorphic 
\begin{align*}
&1)\ J+K: E^{0,1}_{K} \longrightarrow E^{0,1}_{J}  \ \ \ \ \ \ \  2) \ J-K: E^{0,1}_{K} \longrightarrow E^{1,0}_{J} \\
&3)\ J+K: E^{1,0}_{J} \longrightarrow E^{1,0}_{K} \ \ \ \ \ \ \ 4)\ J-K: E^{0,1}_{J} \longrightarrow E^{1,0}_{K} \\
& \ \ \ \ \  \ \  \ \ \ \ \  \ \  \ \ 5)\ [J,K]: E^{0,1}_{K} \longrightarrow E^{1,0}_{K}. 
\end{align*}
\end{prop}

\begin{proof}
The proofs of 1) and 2) are straightforward. To prove 3), let $e \in \Gamma(E^{1,0}_{J})$ satisfy $\nabla^{0,1}e=0$ and consider 
 \begin{equation*} \nabla'^{0,1}(J+K)e= (i+K)\nabla'^{0,1}e= 
\frac{1}{2}(i+K)(\nabla^{0,1}K)Ke.
\end{equation*}
 Since $(i+K)\nabla^{0,1}K=0$, the map given in 3) is holomorphic.
 
 The proof of 4) is similar, and that of 5) follows by composing the maps in 1) and 4) or the maps in 2) and 3).
\end{proof}
Note that Theorem \ref{ThmPHS} immediately follows from the above proposition.

Although we proved the holomorphicity of $[J,K]$ by composing, say, the maps given in 1) and 4), we should stress that it does not depend on the (1,1) condition on the curvature of $\nabla$, $R^{\nabla}$---but only depends on the (1,1) condition on $R^{\nabla'}$: 
  \begin{prop}
  \label{PropBUNM2}
  Let $\nabla$ be a connection on $E \longrightarrow (M,I)$ and let $J$ and $K$  be sections of $\mathcal{C}$ such that $\nabla J=0$ and $K\nabla K=\nabla_{I} K$. Moreover, assume that  $R^{\nabla'}$, where $\nabla'=\nabla + \frac{1}{2}(\nabla K)K$,  is (1,1).
  Then 
\[ [J,K]: E^{0,1}_{K} \longrightarrow E^{1,0}_{K} \]
is holomorphic, where each of the bundles is equipped with the $\overline{\partial}$-operator $\nabla'^{0,1}$. 
  \end{prop}
  \begin{proof}
  Let $e \in \Gamma(E^{0,1}_{K})$ satisfy $\nabla'^{0,1}e=0$ and consider 
  \begin{equation*}
  \nabla'^{0,1}[J,K]e=-(i+K)(\nabla'^{0,1}J)e.
  \end{equation*} 
  Since $\nabla J=0$, this equals 
  \[-\frac{1}{2}(i+K)[(\nabla^{0,1}K)K,J]e,\]
which is zero because 
\[(i+K)\nabla^{0,1}K= \nabla^{0,1}K (i-K)=0.\]
Hence $[J,K]$ is holomorphic. 
  \end{proof}
  As an example, if $(M,g,I)$ is not an SKT manifold then the curvature of $\nabla^{-}$ (Section \ref{SecSKT}) is not (1,1) but the curvature of $(\nabla^{-})'=\nabla^{-} + \frac{1}{2}(\nabla^{-} I)I=\nabla^{Ch}$ is always so. (Note that $I\nabla^{-}I$ still equals $\nabla^{-}_{I}I$.)
\subsubsection{Metric Case}
Considering the setup of Theorem \ref{ThmPHS}, let us now further suppose that $E$ is equipped with a fiberwise metric $g$, $\nabla$ is 
a metric connection and $J$ and $K$ are sections of $\mathcal{T}(E,g)$. 
Then note:
\begin{itemize}
\item The holomorphicity of map 3) in Proposition \ref{PropBUNM}  can be derived from that of 1). The reason is that map 3) equals  $-g^{-1}(J+K)^{t}g$, where $J+K: E^{0,1}_{K} \longrightarrow E^{0,1}_{J}$. Similarly, the holomorphicity of map 4) can be derived from that of 2).
\item The maps 
   \begin{itemize}
   \item[$\cdot$] $[J,K]g^{-1}: E^{*1,0}_{K} \longrightarrow E^{1,0}_{K}$
   \item[$\cdot$] $g[J,K]: E^{0,1}_{K} \longrightarrow E^{*0,1}_{K}$
   \end{itemize}
respectively define holomorphic sections of $\wedge^{2}E^{1,0}_{K}$ and $\wedge^{2}E^{*0,1}_{K}$.  Using Proposition \ref{PropBUNM2}, this still holds  true if $R^{\nabla'}$, but not necessarily $R^{\nabla}$, is $(1,1)$.
\end{itemize}
\subsection{The $M^{\#}$ and $M^{\delta}$ }
\label{SecMDEL}
Consider the setup of Theorem \ref{ThmPHS} of a real rank $2n$ bundle $E \longrightarrow (M,I)$ and the respective parallel and holomorphic sections $J$ and $K$ of $(\mathcal{C}(E),\mathcal{J}^{(\nabla,I)})\longrightarrow (M,I)$. We will now introduce decompositions of the holomorphic subvarieties  $M_{\leq s}$ and $M_{(\leq r,\pm)}$.  
\begin{defi}
 \label{DefMPOUND}
Let $M^{\#}$ stand for any of the following: 
\begin{align*} 
1) \ &  M^{(m_{1},*)}= \{x \in M | \ dimKer(J + K)|_{x}= 2m_{1} \} \\
2) \ &  M^{(*,m_{-1})}= \{x \in M | \ dimKer(J - K)|_{x}= 2m_{-1} \} \\
3) \ & M^{(m_{1},m_{-1})}= M^{(m_{1},*)} \cap  M^{(*,m_{-1})}.
\end{align*}
\end{defi}
\begin{nota} \label{notaMPOUND22}
So far we have introduced the $M_{\leq s}$, $M_{(\leq r,\pm)}$ and $M^{\#}$. We will let $M^{\delta}$ stand for any of these holomorphic subvarieties.
\end{nota}

To relate the $M^{\#}$ to the other subvarieties, we will first decompose $M_{s}$ (see Notation \ref{nota123}) into a disjoint union of some of the $M^{(m_{1},m_{-1})}$: 
 
\begin{prop}
\label{PropMS} 
\begin{itemize}
\item[]
\item[1)]  \[ M_{s}= \bigcup_{m_{1}+m_{-1}=n-s} M^{(m_{1},m_{-1})}.\]

\item[2)] Each $M^{(m_{1},m_{-1})}$ in the above union is open in $M_{s}$.
\item[3)] $M_{(r,+)}=M^{(n-r,*)}$ and $M_{(r,-)}=M^{(*,n-r)}$.
 \end{itemize}
 \end{prop}
 \begin{proof}
 The proof of Part 1) follows from Lemma \ref{LemCOM} given below. While the proof of Part 2) follows from the fact that the dimensions of $ker(J+K)$ and $ker(J-K)$ cannot locally increase. Lastly, the proof of Part 3) is clear.
 \end{proof}
 \begin{lemma}
 \label{LemCOM}
 Let $V$ be an even dimensional real vector space and let $J$ and $K$ be elements of $\mathcal{C}(V)$. The $ker[J,K]=ker(J+K) \oplus ker(J-K)$.
 \end{lemma} 
 
 \begin{proof}
 If we restrict $JK$ to $ker[J,K]$ then it squares to $1$. Hence \[ker[J,K]= W_{1} \oplus W_{-1},\]
 where $JK|_{W_{1}}=1$ and $JK|_{W_{-1}}=-1$.
 The lemma then follows from the fact that $W_{1}=ker(J+K)$ and $W_{-1}=ker(J-K)$.
 \end{proof}
 
 \subsubsection{Metric Case}
 \label{SecTHSMC}
 Let us now further suppose that $E$ is equipped with a fiberwise metric $g$, $\nabla$ is a metric connection and $J$ and $K$ are sections of $\mathcal{T}(E,g)$. The following gives some additional properties of the $M^{\#}$.
 
 \begin{prop}
 \label{PropM'P}
 \begin{enumerate}
  \item[]
  \item[1)] Given $x \in M$, the following is an orthogonal and $J,K-$invariant splitting of $E_{x}$: \[Im[J,K] \oplus ker(J+K) \oplus ker(J-K).\] Moreover the $rank[J,K]=4k$.
  \item[2)] $M$ is a disjoint union of the following open subsets: \[\bigcup_{m_{1}=even}M^{(m_{1},*)} \text{ and } \bigcup_{m_{1}=odd}M^{(m_{1},*)}.\]
  \item[3)] $\bigcup_{m_{1}=even}M^{(m_{1},*)}= \{x \in M | \ J_{x} \text{ and } K_{x} \text{ induce the same orientations} \}. $
  \item[4)] $\bigcup_{m_{1}=odd}M^{(m_{1},*)}= \{x \in M | \ J_{x} \text{ and } K_{x} \text{ induce opposite orientations} \}. $
 \end{enumerate} 
Parts $2)-4)$ are also true if we were to replace $M^{(m_{1},*)}$ with $M^{(*,m_{-1})}$ and $J_{x}$ with $-J_{x}$.
\end{prop} 
\begin{example}
Suppose that $rankE=4$ and $J$ and $K$ induce the same orientations on $E$. Then by the above proposition, $M=M^{(0,0)} \cup M^{(0,2)} \cup M^{(2,0)}$. \qed
\end{example}

 Proposition \ref{PropM'P} follows immediately from the following brief background on the algebraic interaction of two complex structures. 

\textbf{Some Background:} Let $(V,g)$ be an even dimensional real vector space equipped with a positive definite metric and let $J$ and $K \in \mathcal{T}(V,g)$. Consider the orthogonal and $J,K-$invariant splitting: 
 \begin{equation} 
 \label{EqSP}
 V= Im[J,K] \oplus ker(J+K) \oplus ker(J-K).
 \end{equation}
 We then have:

\begin{prop}
\label{PropREP}
\begin{align*}
1) \ & Im[J,K] \text{ is an } \mathbb{H}\text{-module and is 4k real dimensional}.\\
2) \ & J \text{ and } K \text{ induce the same orientation on $V$ if and only if }\frac{dimKer(J+K)}{2} \\ &\text{is even. }
\end{align*}
\end{prop}
\begin{rmk}
The above proposition would not necessarily be true  if we were to assume that $J$ and $K$ are two general elements of $\mathcal{C}(V)$. For instance, take $V=<v_{1},v_{2}>_{\mathbb{R}}$ and define $J$ and $K$ via the equations 
\begin{itemize}
\item $Jv_{1}=v_{2}$, $Jv_{2}=-v_{1}$
\item $Kv_{1}=-rv_{2}$, $Kv_{2}=r^{-1}v_{1}$,  where $r \in \mathbb{R}_{> 0} - {\{1\}} $.
\end{itemize}
 Then $J$ and $K$ are elements of $\mathcal{C}(V)$ that induce opposite orientations on $V$ and yet the $ker(J+K)=0$. Moreover, the $rank[J,K]=2$ and not a multiple of four. Note that $J,K \notin \mathcal{T}(V,g)$ for any metric $g$ because the eigenvalues of $JK$ do not have norm 1.
\end{rmk}

\begin{proof}[Proof of Proposition \ref{PropREP}]
Let us begin by diagonalizing $JK$:
\[V \otimes {\mathbb{C}}= (V_{c_{1}} \oplus \overline{V_{c_{1}}}) \oplus ... \oplus (V_{c_{l}} \oplus \overline{V_{c_{l}}}) \oplus (V_{1}\otimes \mathbb{C}) \oplus (V_{-1}\otimes \mathbb{C}),\]
where $c \in \mathbb{C} - \{\pm1 \}$ satisfies $c\overline{c}=1$, $V_{\pm 1} \subset V$ and $JKv_{\lambda}= \lambda v_{\lambda}$ for $v_{\lambda} \in V_{\lambda}$.

Now set $2e= c + \overline{c}$ and $V_{c} \oplus \overline{V_{c}}=V_{e} \otimes \mathbb{C}$ for $V_{e} \subset V$. We then obtain the following orthogonal and $J,K-$invariant decomposition:  
\[V= V_{e_{1}} \oplus ... \oplus V_{e_{l}} \oplus V_{1} \oplus V_{-1}.\]
Note:
\begin{itemize}
\item $\{J,K\}v=2\epsilon v$, for $v \in V_{\epsilon}$
\item $Im[J,K]= V_{e_{1}} \oplus ... \oplus V_{e_{l}}$, $ker(J+K)=V_{1}$ and $ker(J-K)=V_{-1}$.
\end{itemize}

The claim then is that $V_{e}$---and thus $Im[J,K]$---is an $\mathbb{H}$-module.  The reason is that if we let $J'=\frac{JK-e}{f}|_{V_{e}}$, where $e^{2} +f^{2}=1$, then $(J')^{2}=-1$ and $\{J',K\}=0$ when acting on $V_{e}$. ($\{J',J\}=0$ as well.)

Although the rest of the proof of the proposition follows from this claim, we will now give a more direct proof of the fact that $rank[J,K]=4k$. To see this just note that $g[J,K]:V_{\mathbb{C}}\longrightarrow V^{*}_{\mathbb{C}}$ (where $V_{\mathbb{C}}:= V \otimes \mathbb{C}$) is skew and sends $V^{1,0}_{J} $ to $V^{* 1,0}_{J}$ and $V^{0,1}_{J} $ to $V^{* 0,1}_{J}$.
\end{proof}

\begin{rmk} If we consider the algebra $iD_{\infty}$, generated by two complex structures $J_{0}$ and $K_{0}$ over $\mathbb{R}$, then one may use the above proof to derive its orthogonal representations. (These representations are by definition the ones where $J_{0}$ and $K_{0}$ act by orthogonal transformations with respect to some metric.)

\begin{cor}
The irreducible, orthogonal representations of $iD_{\infty}$ are:
\[ \mathbb{R}[t]/(p) \oplus K_{0}\mathbb{R}[t]/(p), \]
where $J_{0}K_{0}$ acts by $t$ and 
\begin{align*}
1) \ &p= t\pm1 \\
2) \ &p= (t-c)(t-\overline{c}), \text{ for } c \in \mathbb{C} -\{\pm 1\}, c\overline{c}=1.
\end{align*}
\end{cor}

In \cite{Gindi3}, we not only derive the orthogonal representations of $iD_{\infty}$ but the indecomposable ones as well. 
\end{rmk}
Having given some basic properties of the $M^{\delta}$ in the above propositions, we will present our major results about them in Section \ref{SecGTM}. There, we determine lower bounds on the dimensions of the $M^{\delta}$ as well as necessary conditions for there to exist curves in $M$ that lie in certain $M^{\#}$. We will give the applications of some of these results, for the case when $M$ is a bihermitian manifold, in Section \ref{SecBM}. Our present focus is to derive them by first considering in the next section an example of the general setup of Theorem \ref{ThmPHS} where the base manifold is itself the total space of $(\mathcal{C}(E),\mathcal{J}^{(\nabla,I)}) \longrightarrow (M,I)$. 
We will then use this example in Section \ref{SecATVP} to establish a twistor point of  view of the general $M^{\delta}$---by realizing them as the intersection of $K(M)$ with the corresponding $\mathcal{C}^{\delta}$ in $\mathcal{C}$. Some of these $\mathcal{C}^{\delta}$ will be shown to be complex submanifolds of $\mathcal{C}$ and by determining their dimensions and describing their tangent bundles, we will derive the results mentioned above. (We will also describe a metric version of this setup where $\mathcal{C}$ is replaced with $\mathcal{T}$.)

In the next section, we will first consider the twistor spaces $\mathcal{C}(V)$ and $\mathcal{T}(V,g)$ which are associated to vector spaces and then those associated to vector bundles and will be focusing on studying the above properties of the $\mathcal{C}^{\delta}$ and $\mathcal{T}^{\delta}$. 
 
\section{Stratifications of Twistor Spaces}
\label{SecSTS1}
\subsection{Twistor Spaces of Vector Spaces}
\label{SecTVS}
\subsubsection{$\mathcal{C}(V)$-case}
Let $V$ be a $2n$ dimensional real vector space and let $\mathcal{C}:=\mathcal{C}(V)$ be its twistor space with complex structure $I_{\mathcal{C}}$. 
 To obtain an example of the setup of Theorem \ref{ThmPHS}, consider the trivial bundle $E=\mathcal{C} \times V \longrightarrow \mathcal{C}$ along with its trivial connection $d$. By Proposition \ref{PropTHS}, $\phi$, defined by $\phi |_{K}=K$, is a natural holomorphic section of $(\mathcal{C}(E),\mathcal{J}^{(d,I_{\mathcal{C}})})$. As for a parallel section, we will choose the constant section $J$, a fixed element of $\mathcal{C}$, so that $dJ=0$.
 By Theorem \ref{ThmPHS}, we then have
\begin{prop}
Given $J \in \mathcal{C}$, the following are holomorphic subvarieties of $\mathcal{C}$:
\begin{align*}
 1)  \ & \mathcal{C}_{\leq s}(J)= \{K \in \mathcal{C} |\ Rank [J,K] \leq 2s\}\\
 2) \ & \mathcal{C}_{(\leq r,\pm)}(J)= \{K \in \mathcal{C} |\ Rank (J\pm K) \leq 2r\}.
\end{align*}
\end{prop}

We also have the subvarieties $\mathcal{C}^{\#}(J)$ that correspond to the $M^{\#}$ of Definition \ref{DefMPOUND}. More explicitly, $\mathcal{C}^{\#}(J)$ will stand for any of the following: 

\begin{align*} 
1) \ &  \mathcal{C}^{(m_{1},*)}(J)= \{K \in \mathcal{C} | \ dimKer(J + K) = 2m_{1} \} \\
2) \ &  \mathcal{C}^{(*,m_{-1})}(J)= \{K \in \mathcal{C} | \ dimKer(J - K) = 2m_{-1} \} \\
3) \ & \mathcal{C}^{(m_{1},m_{-1})}(J)= \mathcal{C}^{(m_{1},*)} \cap  \mathcal{C}^{(*,m_{-1})}.
\end{align*}

\begin{nota} When referring to the above subvarieties, we will usually drop the 
``$(J)$'' factors and will denote any one of them by $\mathcal{C}^{\delta}$.
\end{nota}
 We will now be studying different properties of the $\mathcal{C}^{\delta}$. In particular, we will show that the $\mathcal{C}^{\#}$ are complex submanifolds that form several stratifications of $\mathcal{C}$ and will determine their dimensions and describe their tangent bundles. 

To accomplish this, we will be using the following holomorphic embedding of $\mathcal{C}$ into the Grassmannians of $n$-planes in $V_{\mathbb{C}}= V \otimes \mathbb{C}$: 

\begin{lemma}
The map 
\begin{align*}
\mu:\  &\mathcal{C} \longrightarrow Gr_{n}(V_{\mathbb{C}}) \\
&K \longrightarrow V^{0,1}_{K},  
 \end{align*}
where $V^{0,1}_{K}$ is the $-i$ eigenspace of $K$, is a holomorphic embedding whose image is open in 
$Gr_{n}(V_{\mathbb{C}})$.
\end{lemma}

\begin{proof}
See for example \cite{Gindi1}. 
\end{proof}

We then have

\begin{prop}
\label{PropCGR}
\begin{align*}
&1) \ \mu(\mathcal{C}_{(\leq r,+)}) = \{W \in Im \mu | \ dim_{\mathbb{C}}(V^{1,0}_{J} \cap W) \geq n-r \}\\
&2) \ \mu(\mathcal{C}^{(m_{1},*)}) = \{W \in Im \mu | \ dim_{\mathbb{C}}(V^{1,0}_{J} \cap W) = m_{1}\} \\
&3) \ \mu(\mathcal{C}^{(*,m_{-1})}) = \{W \in Im \mu | \ dim_{\mathbb{C}}(V^{0,1}_{J} \cap W) = m_{-1}\}.
\end{align*}
Analogous formulas hold for $\mathcal{C}_{(\leq r,-)}$ and $\mathcal{C}^{(m_{1},m_{-1})}$.
\end{prop}
We thus find that $\mu$ maps $\mathcal{C}^{(m_{1},*)}$, $\mathcal{C}^{(*,m_{-1})}$ and $\mathcal{C}_{(\leq r,\pm)}$ to open subsets of either $Gr^{(s)}$ or $Gr^{(\geq s)}$ in $Gr_{n}(V_{\mathbb{C}})$, where $Gr^{(s)}= \{W \in Gr_{n}(V_{\mathbb{C}}) | \ dim_{\mathbb{C}}(V^{0}\cap W)=s \}$ for some $V^{0} \in Gr_{n}(V_{\mathbb{C}})$. $Gr^{(s)}$ is a type of Schubert cell in $Gr_{n}(V_{\mathbb{C}})$ and we will now review some of its properties. 

$\mathbf{Gr^{(s)}:}$
Let $V$ be a real vector space of dimension $2n$ and let $V^{0} \in Gr_{n}(V_{\mathbb{C}})$. As above, define $Gr^{(s)}= \{W \in Gr_{n}(V_{\mathbb{C}}) | \ dim_{\mathbb{C}}(V^{0}\cap W)=s \}$. We will now introduce certain holomorphic charts for $Gr_{n}(V_{\mathbb{C}})$ that will, in particular, be used to show that $Gr^{(s)}$ is a complex submanifold.

To begin, let $W \in Gr^{(s)}$ and split \[V_{\mathbb{C}}=W \oplus W' = (W_{1} \oplus W_{2}) \oplus(W'_{1} \oplus W'_{2}),\]
where $W_{1}= V^{0} \cap W$, $W_{1} \oplus W'_{2}=V^{0}$ and $W_{2}$ and $W'_{1}$ are appropriate complements.

Now consider the corresponding holomorphic chart for $Gr_{n}(V_{\mathbb{C}})$ about $W$: 
\begin{align*}
 \rho:En&d(W,W')   \longrightarrow Gr_{n}(V_{\mathbb{C}}) \\ 
 A & \longrightarrow Graph(A) = \{w+Aw \in V_{\mathbb{C}} | \ w \in W \}. 
\end{align*}
We then have 

\begin{prop}
\label{PropGRAPH}
Let $A = \bordermatrix{~ & W_{1} & W_{2} \cr
                  W'_{1} & a_{1} & a_{2} \cr
                  W'_{2} & a_{3} & a_{4} \cr}  \in End(W,W')$. Then $Graph(A) \cap V^{0}= \{w + Aw \in V_{\mathbb{C}}| \ w \in ker a_{1} \} $ and its dimension equals that of  $ker a_{1}$.
\end{prop}

\begin{proof}
Let $w+Aw \in Graph(A)$ and set $w=w_{1} + w_{2} \in W_{1} \oplus W_{2}$. Then $w+Aw \in V^{0}$ if and only if $ w_{2} +a_{1}w_{1} + a_{2}w_{2}=0$, which in turn is equivalent to $w=w_{1} \in ker a_{1}$.
\end{proof}

If we define 
\[End_{t}(W,W')= \{ \bordermatrix{~ & W_{1} & W_{2} \cr
                  W'_{1} & a_{1} & a_{2} \cr
                  W'_{2} & a_{3} & a_{4} \cr} \in End(W,W') | \ dimKera_{1}=t \}\]

we then have 

\begin{cor} For each $t \in \{0,1,...,s\}$, the map 
\begin{align*}
 End_{t}&(W,W') \longrightarrow Gr^{(t)} \cap Im\rho \\ 
 A & \longrightarrow Graph(A) 
\end{align*}
is well defined and bijective. Moreover, when $t=s$ this map gives a holomorphic chart for $Gr^{(s)}$ about $W$. 
\end{cor}

Using the above corollary, it is straightforward to show:

\begin{cor}
\label{CorGRS}
\begin{align*}
& 1)\ Gr^{(s)} \text{ is a complex submanifold of   }  Gr_{n}(V_{\mathbb{C}}) \text{ of dimension } n^{2} - s^{2}.\\
& 2)\ \overline{Gr^{(s)}} \text { equals } Gr^{(\geq s)} \text{ and is a holomorphic subvariety of dimension } n^{2} - s^{2}. 
\end{align*}
\end{cor}
\textbf{Properties of the $\mathcal{C}^{\delta}$}:
By then combining Proposition \ref{PropCGR} and Corollary \ref{CorGRS}, we obtain

\begin{prop}
\label{PropDIMC}
\begin{align*}
&1) \ \mathcal{C}^{(m_{1},*)} \text{ is a complex submanifold of  } \mathcal{C} \text{ of dimension } n^{2}-m_{1}^{2}. \\
& 2) \  \mathcal{C}_{(\leq r, \pm )} \text{ is a holomorphic subvariety of dimension } n^{2}-(n-r)^{2}.
\end{align*}

An analogous result holds for $\mathcal{C}^{(*,m_{-1})}$. 
\end{prop}
Let us now consider some properties of $\mathcal{C}^{(m_{1},m_{-1})}= \mathcal{C}^{(m_{1},*)}\cap \mathcal{C}^{(*,m_{-1})}$. 
\begin{prop} $\mathcal{C}^{(m_{1},m_{-1})}$ is nonempty if and only if $m_{1} + m_{-1} \leq n$.
\end{prop}
\begin{proof}
If $K \in \mathcal{C}^{(m_{1},m_{-1})}$ then by Lemma \ref{LemCOM}, $dimKer[J,K]= 2(m_{1}+m_{-1}) \leq 2n.$

Conversely, given $m_{1}, m_{-1} \in \mathbb{Z}_{\geq 0}$ such that $m_{1} + m_{-1} \leq n$, we will define a $K \in \mathcal{C}^{(m_{1},m_{-1})}$ as follows. First consider the $J$-invariant splitting: 
\[V= \bigoplus_{i\in \{1,2,...,l\}} <v_{i},Jv_{i}> \oplus V_{1} \oplus V_{-1},\]
where $dimV_{1}=2m_{1}$ and $dimV_{-1}=2m_{-1}$.

Now define $K \in \mathcal{C}$ by setting 
\begin{itemize}
\item $Kv_{i}= -rJv_{i}$ and $KJv_{i}=r^{-1}v_{i}$, where $r \in \mathbb{R} - \{0,\pm1\}$
\item $Kw_{1}=-Jw_{1}$ and $Kw_{-1}=Jw_{-1}$, $\forall \ w_{1} \in V_{1}$ and $w_{-1} \in V_{-1}$.   
\end{itemize}
So defined, one may check that $K$ is indeed an element of $\mathcal{C}^{(m_{1},m_{-1})}$.
\end{proof}

Supposing that $m_{1} + m_{-1} \leq n$, we will now show that $\mathcal{C}^{(m_{1},*)}$ and $\mathcal{C}^{(*,m_{-1})}$ intersect transversally, thus proving, in particular, that $\mathcal{C}^{(m_{1},m_{-1})}$  is a complex manifold. To show this, we will first describe $T_{K}\mathcal{C}^{(m_{1},*)}$ and $T_{K} \mathcal{C}^{(*,m_{-1})}$ in $T_{K}\mathcal{C}= \mathfrak{gl}_{ \{K\}}:= \{A \in \mathfrak{gl}(V)| \ \{A,K\}=0\}.$
\begin{prop}
\label{PropTANC}
\begin{align*}
1)& \ T_{K}\mathcal{C}^{(m_{1},*)}= \{A \in \mathfrak{gl}_{\{K\}} | \ A: Ker(J+K) \rightarrow Im(J+K) \} \\
2)& \ T_{K}\mathcal{C}^{(*,m_{-1})}= \{A \in \mathfrak{gl}_{\{K\}} | \ A: Ker(J-K) \rightarrow Im(J-K) \} \\
3)& \ \text{If } m_{1} + m_{-1} \leq n \text{ then } \mathcal{C}^{(m_{1},*)} \text{ and } \mathcal{C}^{(*,m_{-1})} \text{ intersect transversally. }
\end{align*}
\end{prop}

\begin{proof}
To prove Part 1) of the proposition, let $K \in \mathcal{C}^{(m_{1},*)}$ and $v \in ker(J+K)$ and suppose $K(t)$ is a curve in $\mathcal{C}^{(m_{1},*)}$ that satisfies $K(0)=K$.

As the rank of $J+K(t)$ is independent of $t$, one may extend $v$ to a curve $v(t)$ in $V$ so that 
\[(J+K(t))v(t)=0. \]
Taking $\frac{d}{dt}|_{t=0}$ of the above expression gives  
\[K'v=-(J+K)v'(0), \]
which shows that 
\[T_{K}\mathcal{C}^{(m_{1},*)} \subset \{A \in \mathfrak{gl}_{\{K\}} | \ A: Ker(J+K) \rightarrow Im(J+K) \}. \]
That these subspaces are indeed equal then follows from the fact that they have the same dimensions.

The proof of Part 2) of the proposition is similar and that of Part 3) is straightforward.

\end{proof}

\begin{cor} For $m_{1}+m_{-1}\leq n$,
$\mathcal{C}^{(m_{1},m_{-1})}$ is a complex submanifold of dimension $n^{2} - m_{1}^{2}- m^{2}_{-1}$.
\end{cor}
Since by Proposition \ref{PropMS}  \[\mathcal{C}_{s}= \bigcup_{m_{1}+m_{-1}=n-s} \mathcal{C}^{(m_{1},m_{-1})},\]
it follows that $\mathcal{C}_{s}$ is a disjoint union of complex submanifolds of varied dimensions.
 
     Lastly note that by then using the different $\mathcal{C}^{\#}$ we can stratify $\mathcal{C}$ in several ways, i.e., they can be used to decompose $\mathcal{C}$ into disjoint unions of complex submanifolds. 
     
\subsubsection{$\mathcal{T}$-case}
Let $(V,g)$ be a $2n$ dimensional real vector space with a positive definite metric and let $\mathcal{T}:=\mathcal{T}(V,g)$ be the associated twistor space. Similar to the $\mathcal{C}-$case of the previous section, we have

\begin{prop}
\label{PropTHSJ}
Given $J \in \mathcal{T}$, the following are holomorphic subvarieties of $\mathcal{T}$:
\begin{align*}
 1)  \ & \mathcal{T}_{\leq s}(J)= \{K \in \mathcal{T} |\ Rank [J,K] \leq 2s\}\\
 2) \ & \mathcal{T}_{(\leq r,\pm)}(J)= \{K \in \mathcal{T} |\ Rank (J\pm K) \leq 2r\}.
\end{align*}
\end{prop}
We also have the subvarieties $\mathcal{T}^{\#}(J)$ that correspond to the $M^{\#}$ of Definition \ref{DefMPOUND} and which will stand for $\mathcal{T}^{(m_{1},*)}(J)$, $\mathcal{T}^{(*,m_{-1})}(J)$ and $\mathcal{T}^{(m_{1},m_{-1})}(J)$.
\begin{nota} When referring to the above subvarieties, we will usually drop the 
``$(J)$'' factors and will denote any one of them by $\mathcal{T}^{\delta}$. 
\end{nota}

We will now study the $\mathcal{T}^{\delta}$ in two ways. The first will be to embed them into a certain space of maximal isotropics associated to $V_{\mathbb{C}}:=V\otimes \mathbb{C}$ and the second will be to study them directly inside of $\mathcal{T}$ by using special charts.

\textbf{Maximal Isotropics:}\\ 
To begin, let \[MI(V_{\mathbb{C}})= \{W \in Gr_{n}(V_{\mathbb{C}}) | \ g(w_{1},w_{2})=0,  \ \forall \ w_{1},w_{2} \in W\} \]
be the space of maximal isotropics in $V_{\mathbb{C}}$.  Considering it as a complex submanifold of $Gr_{n}(V_{\mathbb{C}})$, we have

\begin{lemma}
\label{LemGTR}
The map 
\begin{align*}
\mu:\  &\mathcal{T} \longrightarrow MI(V_{\mathbb{C}}) \\
&K \longrightarrow V^{0,1}_{K}, 
 \end{align*}
where $V^{0,1}_{K}$ is the $-i$ eigenspace of $K$, is a biholomorphism.
\end{lemma}

\begin{proof}
See \cite{Gindi1}. 
\end{proof}

\begin{prop}
\label{PropMUT}
\begin{align*}
&1) \ \mu(\mathcal{T}_{(\leq r,+)}) = \{W \in MI(V_{\mathbb{C}}) | \ dim_{\mathbb{C}}(V^{1,0}_{J} \cap W) \geq n-r \}\\
&2) \ \mu(\mathcal{T}^{(m_{1},*)}) = \{W \in MI(V_{\mathbb{C}}) | \ dim_{\mathbb{C}}(V^{1,0}_{J} \cap W) = m_{1}\} \\
&3) \ \mu(\mathcal{T}^{(*,m_{-1})}) = \{W \in MI(V_{\mathbb{C}}) | \ dim_{\mathbb{C}}(V^{0,1}_{J} \cap W) = m_{-1}\}.
\end{align*}
Analogous formulas hold for $\mathcal{T}_{(\leq r,-)}$ and $\mathcal{T}^{(m_{1},m_{-1})}$.
\end{prop}

It follows that $\mu$ maps $\mathcal{T}^{(m_{1},*)}$, $\mathcal{T}^{(*,m_{-1})}$ and $\mathcal{T}_{(\leq r,\pm)}$ to either $MI^{(s)}$ or $MI^{(\geq s)}$ in $MI(V_{\mathbb{C}})$, where $MI^{(s)}= \{W \in MI(V_{\mathbb{C}}) | \ dim_{\mathbb{C}}(V^{0}\cap W)=s \}$ for some $V^{0} \in MI(V_{\mathbb{C}})$. $MI^{(s)}$ is a type of Schubert cell in $MI(V_{\mathbb{C}})$ and we will now review some of its properties. 

$\mathbf{MI^{(s)}:}$
As above let $(V,g$) be a real vector space of dimension $2n$ with a positive definite metric and for $V^{0} \in MI(V_{\mathbb{C}})$, define  $MI^{(s)}= \{W \in MI(V_{\mathbb{C}}) | \ dim_{\mathbb{C}}(V^{0}\cap W)=s \}$.

To study the $MI^{(s)}$, we will begin by defining certain holomorphic charts for $MI(V_{\mathbb{C}})$ that will be used, in particular, to show that the $MI^{(s)}$ are complex submanifolds.

If we let $W \in MI^{(s)}$ we then have: 

\begin{prop}
One may split \[V_{\mathbb{C}}=W \oplus W' = (W_{1} \oplus W_{2}) \oplus(\overline{W_{1}} \oplus W'_{2}),\]
where 
\begin{itemize}

\item $W'$ is a maximal isotropic
\item $W_{1}= V^{0} \cap W$ and $W_{1} \oplus W'_{2}=V^{0}$
\item $g(\overline{w_{1}}, w_{2})=0$,  $\forall \ \overline{w_{1}} \in \overline{W_{1}}$ and  $w_{2} \in W_{2}$.
\end{itemize}
\end{prop}

\begin{proof}
Using Lemma \ref{LemGTR}, let $V^{0}=V^{0,1}_{J}$ and $W=V^{0,1}_{K}$, where $J,K \in \mathcal{T}$ satisfy $dimKer(J-K)=2s$.

Consider then the following orthogonal and $J,K$-invariant splitting:
\[ V= \tilde{V} \oplus ker(J-K),\]
where $\tilde{V}=Im(J-K)$.

If we complexify, we may further split 

\[\tilde{V}\otimes \mathbb{C}= (\tilde{V})^{0,1}_{J} \oplus (\tilde{V})^{0,1}_{K} \]

and 

\[ker(J-K)\otimes \mathbb{C} = V^{0,1}_{J} \cap V^{0,1}_{K} \oplus V^{1,0}_{J} \cap V^{1,0}_{K}. \]

Thus \[V_{\mathbb{C}}= (V^{0,1}_{J} \cap V^{0,1}_{K} \oplus (\tilde{V})^{0,1}_{K} ) \oplus  (V^{1,0}_{J} \cap V^{1,0}_{K} \oplus (\tilde{V})^{0,1}_{J}),\]

which satisfies the conditions listed in the proposition.
\end{proof}

Given the above splitting for $V_{\mathbb{C}}$, consider the corresponding holomorphic chart for $MI(V_{\mathbb{C}})$ about $W$: 
\begin{align*}
 \rho:En&d^{g}(W,W')   \longrightarrow MI(V_{\mathbb{C}}) \\ 
 A & \longrightarrow Graph(A) = \{w+Aw \in V_{\mathbb{C}} | \ w \in W \},
\end{align*}
where $End^{g}(W,W')=\{A \in End(W,W') | \ g(Aw,\tilde{w})=-g(w,A\tilde{w}) \ \forall w,\tilde{w} \in W\}. $

We will now describe how each $Graph(A)$ intersects $V^{0}$: 

\begin{prop}
Given $A = \bordermatrix{~ & W_{1} & W_{2} \cr
                  \overline{W_{1}} & a_{1} & a_{2} \cr
                  W'_{2} & a_{3} & a_{4} \cr}  \in End^{g}(W,W')$, the $Graph(A) \cap V^{0}= \{w + Aw \in V_{\mathbb{C}}| \ w \in ker a_{1} \} $ and its dimension equals that of  $ker a_{1}$.
\end{prop}

\begin{proof}
This follows from Proposition \ref{PropGRAPH}. 
\end{proof}

If we define 
\[End^{g}_{t}(W,W')= \{ \bordermatrix{~ & W_{1} & W_{2} \cr
                  \overline{W_{1}} & a_{1} & a_{2} \cr
                  W'_{2} & a_{3} & a_{4} \cr} \in End^{g}(W,W') | \ dimKera_{1}=t \}\]

then by the above proposition we have: 

\begin{cor} The map 
\begin{align*}
 End^{g}_{t}&(W,W') \longrightarrow MI^{(t)} \cap Im\rho \\ 
 A & \longrightarrow Graph(A) 
\end{align*}
is well defined and bijective. Moreover, when $t=s$ this map gives a holomorphic chart for $MI^{(s)}$ about $W$. 
\end{cor}
Note that in appropriately chosen bases, 
        $ A= \bordermatrix{~ & W_{1} & W_{2} \cr
                  \overline{W_{1}} & a_{1} & a_{2} \cr
                  W'_{2} & a_{3} & a_{4} \cr} \\ \in End^{g}(W,W')$
is a skew matrix. Hence 

\[ End^{g}(W,W') = \left\{ 
  \begin{array}{l l}
  
  \bigcup_{t \in \{0,2,4,...,s\} }End^{g}_{t}(W,W')    
     & \quad \text{if $s$ is even}\\
    \bigcup_{t \in \{1,3,5,...,s\} }End^{g}_{t}(W,W') & \quad \text{if $s$ is odd}.
  \end{array} \right.\]                  
            
It is then straightforward to prove the following proposition:

\begin{prop}
\label{PropMIS}
\begin{align*}
& 1)\ MI(V_{\mathbb{C}}) \text{ is a disjoint union of the following two open subsets: } \\ &\ \ \bigcup_{t=even}MI^{(t)} \text{ and } \bigcup_{t=odd}MI^{(t)}.\\ 
& 2)\ MI^{(s)} \text{ is a complex submanifold of dimension } \frac{n(n-1)-s(s-1)}{2}.\\
& 3)\ \overline{MI^{(s)}} \text{ equals } \bigcup_{k\in \mathbb{Z}_{\geq 0}} MI^{(s+2k)} \text{ and is a holomorphic } \text{ subvariety of}\\ 
&\ \ \  \text{dimension } \frac{n(n-1)-s(s-1)}{2}.\\
& 4)\  MI^{(\geq s)}=  \overline{MI^{(s)}} \cup \overline{MI^{(s+1)}} \text{ and is a holomorphic subvariety of } MI(V_{\mathbb{C}}).
\end{align*}
\end{prop}
\textbf{Properties of the $\mathcal{T}^{\delta}$:} 
If we return to the setup of $(V,g)$ with a fixed element $J \in \mathcal{T}$ then by Proposition \ref{PropMUT} we have a result analogous to Proposition \ref{PropMIS} but for $\mathcal{T}$ instead of $MI(V_{\mathbb{C}})$. For future reference we provide the details:

\begin{prop}
\label{PropTS}
\begin{align*}
& 1)\ \mathcal{T} \text{ is a disjoint union of the following two open subsets: } \\ &\ \ \bigcup_{m_{1}=even}\mathcal{T}^{(m_{1},*)} \text{ and } \bigcup_{m_{1}=odd}\mathcal{T}^{(m_{1},*)}.\\ 
& 2)\ \mathcal{T}^{(m_{1},*)} \text{ is a complex submanifold of dimension } \frac{n(n-1)-m_{1}(m_{1}-1)}{2}.\\
& 3)\ \overline{\mathcal{T}^{(m_{1},*)}} \text{ equals } \bigcup_{k\in \mathbb{Z}_{\geq 0}} \mathcal{T}^{(m_{1}+2k,*)}  \text{ and is a holomorphic subvariety of }\\
& \ \ \text{ dimension } \frac{n(n-1)-m_{1}(m_{1}-1)}{2}.\\
& 4)\  \mathcal{T}_{(\leq r, +)}=  \overline{\mathcal{T}^{(n-r,*)}} \cup \overline{\mathcal{T}^{(n-r+1,*)}} \text{ and is a holomorphic subvariety of } \mathcal{T}.
\end{align*}
Analogous results hold for $\mathcal{T}_{(\leq r, -)}$ and $\mathcal{T}^{(*,m_{-1})}$.
\end{prop}

Another way to prove Part 1) of the above proposition is to use the following:

\begin{prop}\ \
\begin{enumerate}
\item[1)] $\bigcup_{m_{1}=even}\mathcal{T}^{(m_{1},*)}= \{K \in \mathcal{T} | \ J \text{ and } K \text{ induce the same orientations} \}. $
\item[2)] $\bigcup_{m_{1}=odd}\mathcal{T}^{(m_{1},*)}= \{K \in \mathcal{T} | \ J \text{ and } K \text{ induce opposite orientations} \}. $
\end{enumerate}
The above holds true if we were to replace $\mathcal{T}^{(m_{1},*)}$ with $\mathcal{T}^{(*,m_{-1})}$ and $J$ with $-J$.
\end{prop}
\begin{proof} The proof follows from Proposition \ref{PropM'P}.
\end{proof}
\begin{rmk}
Below, we will present an alternative derivation of (the 
entire) Proposition \ref{PropTS} without using maximal isotropics.
\end{rmk}
Let us now consider some properties of $\mathcal{T}^{(m_{1},m_{-1})}= \mathcal{T}^{(m_{1},*)} \cap \mathcal{T}^{(*,m_{-1})}$.  

\begin{prop}
$\mathcal{T}^{(m_{1},m_{-1})}$ is nonempty if and only if $n-m_{1}-m_{-1}=2k$ $(k \in \mathbb{Z}_{\geq 0})$.
\end{prop}
\begin{proof}
If $K \in \mathcal{T}^{(m_{1},m_{-1})}$ then by Lemma \ref{LemCOM}, $dimKer[J,K]=2(m_{1}+m_{-1})$ and by Proposition \ref{PropREP},  $rank[J,K]=4k$. Hence $2n=4k+2(m_{1}+m_{-1})$. The proof of the rest of the proposition is straightforward. 
\end{proof}
Supposing that $n-m_{1}-m_{-1}=2k$, we will now show that $\mathcal{T}^{(m_{1},*)}$ and $\mathcal{T}^{(*,m_{-1})}$ intersect transversally. To do so, we will first describe $T_{K}\mathcal{T}^{(m_{1},*)}$ and $T_{K}\mathcal{T}^{(*,m_{-1})}$ in $T_{K}\mathcal{T}= \mathfrak{o}_{\{ K\}}:=\{A \in \mathfrak{o}(V,g) | \ \{A,K\}=0\}$. If we split $V=Im[J,K] \oplus ker(J+K) \oplus ker(J-K)$ and let $P_{0}$, $P_{1}$ and $P_{-1}$ be the corresponding projection operators, we then have

\begin{prop}
\label{PropTANTAU}
\begin{align*}
&1) \ T_{K}\mathcal{T}^{(m_{1},*)}= \{A \in \mathfrak{o}_{\{K \}} | \ P_{1}AP_{1}=0\} \\
&2) \ T_{K}\mathcal{T}^{(*,m_{-1})}= \{A \in \mathfrak{o}_{\{K \}} | \ P_{-1}AP_{-1}=0\} \\
&3) \ \text{If } n-m_{1}-m_{-1}=2k, \text{ for } k \in \mathbb{Z}_{\geq 0}, \text{ then } \mathcal{T}^{(m_{1},*)} \text{ and } \mathcal{T}^{(*,m_{-1})}   \\ 
& \ \ \text{ intersect transversally}. 
\end{align*}
\end{prop}
\begin{proof}The proof is similar to that of Proposition \ref{PropTANC}.
\end{proof}
\begin{cor}For $n-m_{1}-m_{-1}=2k$ $(k \in \mathbb{Z}_{\geq 0})$, $\mathcal{T}^{(m_{1},m_{-1})} $ is a complex submanifold of dimension $\frac{1}{2}(n(n-1)-m_{1}(m_{1}-1)-m_{-1}(m_{-1}-1)).$
\end{cor}
Since \[\mathcal{T}_{s}= \bigcup_{m_{1}+m_{-1}=n-s} \mathcal{T}^{(m_{1},m_{-1})},\]
we find that $\mathcal{T}_{s}$ is a disjoint union of complex submanifolds of varied  dimensions.

\textbf{Other Charts for $\mathcal{T}(V,g)$:}
We will now describe a $C^{\infty}$ chart for $\mathcal{T}$ about $K \in \mathcal{T}^{(m_{1},m_{-1})}(J)$ that can be used to derive Proposition \ref{PropTS} without using maximal isotropics.

 There is of course the standard chart that is induced from the map 
\begin{align*}
\mathfrak{o}&_{\{K\}} \longrightarrow \mathcal{T} \\
&A \longrightarrow exp(A)Kexp(-A),  
 \end{align*} 
  where $\mathfrak{o}_{\{K\}}= \{A \in \mathfrak{o}(V,g)| \ 
  \{A,K\}=0 \}$.
  However, this chart has the disadvantage that it does not allow us  to immediately determine in which $\mathcal{T}^{(m_{1}',m_{-1}')}$ $exp(A)Kexp(-A)$ lies. Instead, let us use $J$ and $K$ to decompose $\mathfrak{o}(V,g)$ into
 \[ \mathfrak{u}_{J} + \mathfrak{u}_{K} \oplus \mathfrak{o}_{\{J\}}  
 \cap \mathfrak{o}_{\{K\}},\] where $\mathfrak{u}_{K}= \{A \in \mathfrak{o}(V,g)| \ [A,K]=0 \}.$ (Note that this is an orthogonal splitting of $\mathfrak{o}(V,g)$, where the metric used is $-tr$.) Letting $\mathcal{D}_{J}$ be a complement to $\mathfrak{u}_{J} \cap \mathfrak{u}_{K}$ in $\mathfrak{u}_{J}$, consider then the map 
\begin{align*}
\psi: \mathcal{D}_{J} & \oplus \mathfrak{o}_{\{J\}}  
 \cap \mathfrak{o}_{\{K\}}  \longrightarrow \mathcal{T} \\
 A & + B \longrightarrow exp(A)\cdot exp(B)\cdot K,  
 \end{align*} 
 where $exp(B)\cdot K:= exp(B)Kexp(-B)$.
  This map is a local diffeomorphism from a neighborhood about the origin in the domain to one about the point $K$ in the range. We will use it to define $C^{\infty}$ charts for the $\mathcal{T}^{\#}(J)$ about $K$ as follows: First consider the orthogonal and $J,K$-invariant splitting of  $V=V_{0} \oplus V_{1} \oplus V_{-1}$, where $V_{\pm1}=ker(J \pm K)$, and the corresponding splitting of $\mathfrak{o}_{\{J\}} \cap \mathfrak{o}_{\{K\}}= \mathfrak{o}_{\{J\}}  
 \cap \mathfrak{o}_{\{K\}}|_{V_{0}} \oplus \mathfrak{o}_{\{J\}}|_{V_{1}} \oplus \mathfrak{o}_{\{J\}}|_{V_{-1}}$. We then have: 
 
 \begin{prop}
 Suppose $K \in \mathcal{T}^{(m_{1},m_{-1})}(J)$ and let $A \in \mathcal{D}_{J}$ and $B=B_{0}+B_{1}+B_{-1} \in \mathfrak{o}_{\{J\}}  
 \cap \mathfrak{o}_{\{K\}}|_{V_{0}} \oplus \mathfrak{o}_{\{J\}}|_{V_{1}} \oplus \mathfrak{o}_{\{J\}}|_{V_{-1}} $. There exists a neighborhood $N$ in $\mathfrak{o}_{\{J\}} 
 \cap \mathfrak{o}_{\{K\}}$ about the origin such that if $B \in N$ then 
 \begin{align*}
 & 1)\  \psi(A+B) \in \mathcal{T}^{(m_{1},*)}(J) \text{ if and only if } B_{1}=0
 \\
  & 2)\  \psi(A+B) \in \mathcal{T}^{(*,m_{-1})}(J) \text{ if and only if }  B_{-1}=0\\
& 3)\  \psi(A+B) \in \mathcal{T}^{(m_{1},m_{-1})}(J) \text{ if and only if }  B_{1}=0 \text{ and } B_{-1}=0.
 \end{align*}
 \end{prop}
 
 \begin{proof}
  First note that since $A \in \mathfrak{u}_{J}$, the $dimKer(J \pm exp(A)\cdot exp(B)\cdot K)= dimKer(J \pm exp(B) \cdot K)$. Focusing on the proof of Part 1) of the proposition, let us split $J= J_{0} \oplus J_{1}\oplus J_{-1}$ according to the decomposition of $V=V_{0} \oplus V_{1} \oplus V_{-1}$. It then follows that for small enough $B \in \mathfrak{o}_{\{J\}} 
 \cap \mathfrak{o}_{\{K\}}$, the $ker(J + exp(B) \cdot K)=   ker(J_{1} + exp(B_{1}) \cdot K_{1})$. Hence the $dimKer(J + exp(B) \cdot K)=2m_{1}$ if and only if $J_{1}= -exp(B_{1})\cdot K_{1}$. As $K_{1}=-J_{1}$ and $B_{1} \in \mathfrak{o}_{\{J\}}|_{V_{1}}$, this is equivalent to $exp(2B_{1})=1$ (as endomorphisms of $V_{1}$). For small enough $B$, this in turn holds if and only if $B_{1}=0$. The other parts of the proposition are proved similarly.  
  \end{proof}
    This then defines charts for the $\mathcal{T}^{\#}(J)$ that together with Proposition \ref{PropTHSJ} can be used to derive Proposition \ref{PropTS} without using maximal isotropics. 
\subsection{Twistors of Bundles}
\label{SecTB23}
 We will now introduce the holomorphic subvarieties $\mathcal{C}^{\delta}$ and $\mathcal{T}^{\delta}$ for the case when the twistor spaces are associated to bundles. Their properties will be used in Section \ref{SecGTM} to derive our main results about the general $M^{\delta}$.   
\subsubsection{$\mathcal{C}$-case}
Let $E \longrightarrow (M,I)$  be a real rank $2n$ bundle fibered over a complex manifold and equipped with 
\begin{itemize}
\item a connection $\nabla$ that has (1,1) curvature
\item a section $J$ of $\pi: \mathcal{C}(E) \longrightarrow M$ satisfying $\nabla J=0$. 
\end{itemize}
 We will use $J$ to decompose $(\mathcal{C}(E),\mathcal{J}^{(\nabla,I)})$ into complex submanifolds and holomorphic subvarieties as follows. First consider the bundle $\pi^{*}E$ with the connection $\pi^{*}\nabla$ that has (1,1) curvature with respect to $\mathcal{J}^{(\nabla,I)}$.  
By Proposition \ref{PropTHS}, $\phi$, defined by $\phi |_{K}=K$, is a natural  holomorphic section of $(\mathcal{C}(\pi^*E),\mathcal{J}^{(\pi^{*}\nabla,\mathcal{I})}) \longrightarrow \mathcal{C}(E),$ where $\mathcal{I}=\mathcal{J}^{(\nabla,I)}$. Since $\pi^{*}J$ is a parallel section, by Theorem \ref{ThmPHS} we obtain the following holomorphic subvarieties in $\mathcal{C}:=\mathcal{C}(E)$.

\begin{prop}
\label{PropCEHS} 
Let $J \in \Gamma(\mathcal{C})$ satisfy $\nabla J=0$. The following are holomorphic subvarieties of $(\mathcal{C},\mathcal{J}^{(\nabla,I)})$:
\begin{align*}
 1)  \ & \mathcal{C}_{\leq s}(J)= \{K \in \mathcal{C}  | \ Rank [J,K] \leq 2s\}\\
 2) \ & \mathcal{C}_{(\leq r,\pm)}(J)= \{K \in \mathcal{C} | \ Rank (J\pm K) \leq 2r\}.
\end{align*}
\end{prop} 
We also have the subvarieties $\mathcal{C}^{\#}(J)$ that correspond to the $M^{\#}$ of Definition \ref{DefMPOUND} and which will stand for $\mathcal{C}^{(m_{1},*)}(J)$, $\mathcal{C}^{(*,m_{-1})}(J)$ and $\mathcal{C}^{(m_{1},m_{-1})}(J)$.
\begin{nota} When referring to the above subvarieties, we will usually drop the 
``$(J)$'' factors and will denote any one of them by $\mathcal{C}^{\delta}$.
\end{nota} 

Using the fact that the $\mathcal{C}^{\#}$ are $C^{\infty}$ fiber bundles together with the results of Propositions \ref{PropDIMC} and \ref{PropCEHS}, we arrive at the following:
\begin{prop}
\label{PropCODCE}
\ \ 
\begin{enumerate} 
  \item[1)] The $\mathcal{C}^{\#}$ are complex submanifolds of 
                $\mathcal{C}$  and 
           have the following codimensions: 
      \begin{enumerate}
               \item[a)]  $codim_{\mathbb{C}} \mathcal{C}^{(m_{1},*)}= 
                              m_{1}^{2}$ \\
                \item[b)]  $codim_{\mathbb{C}} \mathcal{C}^{(*,m_{-1})}= 
                               m_{-1}^{2}$\\
                \item[c)]  $codim_{\mathbb{C}} \mathcal{C}^{(m_{1},m_{-1})}=     
                               m_{1}^{2}+m_{-1}^{2}$.
      \end{enumerate}       
\item[]                 
\item[2)] $\mathcal{C}_{(\leq r,\pm)}$ is a holomorphic subvariety of $\mathcal{C}$  of codimension $(n-r)^{2}$.
\end{enumerate}
\end{prop}

We can also describe $T_{K}\mathcal{C}^{\#}(J)$ as follows. First recall that  $\nabla$ induces a splitting of $T\mathcal{C}$ into $V\mathcal{C} \oplus H^{\nabla}\mathcal{C}$, where $V_{K}\mathcal{C}=T_{K}\mathcal{C}(E_{\pi(K)})$ and $H^{\nabla}\mathcal{C}$ is a certain horizontal distribution (see  \cite{Gindi1}). Since $\nabla J=0$, we have

\begin{prop} 
\label{PropTCSPLIT}
$T_{K}\mathcal{C}^{\#}(J)= V_{K}\mathcal{C}^{\#}(J) \oplus  H_{K}^{\nabla}\mathcal{C}.$
\end{prop}
Note that $V_{K}\mathcal{C}^{\#}(J)= T_{K}\mathcal{C}(E_{\pi(K)})^{\#}(J)$ was already described in Proposition \ref{PropTANC}.
\begin{proof}[Proof of Proposition \ref{PropTCSPLIT}]
By \cite{Gindi1}, a general element of $H_{K}^{\nabla}\mathcal{C}$ is given by 
$\frac{dK(t)}{dt}|_{t=0}$, where $K(t)$ is the parallel translate of $K=K(0)$ (using $\nabla$) along some curve in $M$. The proof of the proposition then follows from the fact that since $\nabla J =0$, $K(t) \in \mathcal{C}^{(m_{1},m_{-1})}(J)$ for some $(m_{1},m_{-1})$. 
\end{proof}

\subsubsection{$\mathcal{T}$-case}
Let $(E,g) \longrightarrow (M,I)$  be a real rank $2n$ bundle fibered over a complex manifold and equipped with a fiberwise metric. Also let \begin{itemize}
\item  $\nabla$ be a metric connection on $E$ that has (1,1) curvature and
\item $J$ be a section of $\pi: \mathcal{T}(E,g) \longrightarrow M$ that satisfies  $\nabla J=0$. 
\end{itemize}
If we consider $\mathcal{T}:=\mathcal{T}(E,g)$ with its complex structure $\mathcal{J}^{(\nabla,I)}$ then, similar to the $\mathcal{C}-$case of the previous section, we have
\begin{prop} 
\label{PropTESUB}
Let $J \in \Gamma(\mathcal{T})$ satisfy $\nabla J=0$. The following are holomorphic subvarieties of $(\mathcal{T},\mathcal{J}^{(\nabla,I)})$:
\begin{align*}
 1)  \ & \mathcal{T}_{\leq s}(J)= \{K \in \mathcal{T}  | \ Rank [J,K] \leq 2s\}\\
 2) \ & \mathcal{T}_{(\leq r,\pm)}(J)= \{K \in \mathcal{T} | \ Rank (J\pm K) \leq 2r\}.
\end{align*}
\end{prop} 
We also have the subvarieties $\mathcal{T}^{\#}(J)$ that correspond to the $M^{\#}$ of Definition \ref{DefMPOUND} and which will stand for $\mathcal{T}^{(m_{1},*)}(J)$, $\mathcal{T}^{(*,m_{-1})}(J)$ and $\mathcal{T}^{(m_{1},m_{-1})}(J)$.
\begin{nota} When referring to the above subvarieties, we will usually drop the 
``$(J)$'' factors and will denote any one of them by $\mathcal{T}^{\delta}$.
\end{nota} 

Note it is straightforward to show that the $\mathcal{T}^{\#}$ are in fact  
 $C^{\infty}$ fiber bundles. Using this together with the results of Propositions \ref{PropTS} and \ref{PropTESUB}, we arrive at
\begin{prop}
\label{PropCODTAU}
\ \ 
\begin{enumerate} 
  \item[1)] The $\mathcal{T}^{\#}$ are complex submanifolds of $\mathcal{T}$  and have the following codimensions: 
      \begin{enumerate}
               \item[a)]  $codim_{\mathbb{C}} \mathcal{T}^{(m_{1},*)}= \frac{m_{1}(m_{1}-1)}{2}$ \\
                \item[b)]  $codim_{\mathbb{C}} \mathcal{T}^{(*,m_{-1})}= \frac{m_{-1}(m_{-1}-1)}{2}$\\
                \item[c)]  $codim_{\mathbb{C}} \mathcal{T}^{(m_{1},m_{-1})}=     
                       \frac{m_{1}(m_{1}-1) + m_{-1}(m_{-1}-1) }{2}$.
      \end{enumerate}      
\item[]
\item[2)] $\overline{\mathcal{T}^{(m_{1},*)}}= \bigcup_{k\in \mathbb{Z}_{\geq 0}} \mathcal{T}^{(m_{1}+2k,*)}$ and is a holomorphic subvariety of codimension $\frac{m_{1}(m_{1}-1)}{2}$. 
\item[3)] $\mathcal{T}_{(\leq r,+)}=\mathcal{T}^{(\geq n-r, *)}= \overline{\mathcal{T}^{(n-r,*)}}\cup \overline{\mathcal{T}^{(n-r+1,*)}}. $
\end{enumerate}
Parts 2) and 3) are also true if we were to replace $\mathcal{T}_{(\leq r,+)}$ with $\mathcal{T}_{(\leq r,-)}$ and $\mathcal{T}^{(m_{1},*)}$ with $\mathcal{T}^{(*,m_{-1})}.$
\end{prop}

Let us now describe $T_{K}\mathcal{T}^{\#}(J)$. First recall that  $\nabla$ induces a splitting of $T\mathcal{T}$ into $V\mathcal{T} \oplus H^{\nabla}\mathcal{T}$, where $V_{K}\mathcal{T}=T_{K}\mathcal{T}(E_{\pi(K)})$ and $H^{\nabla}\mathcal{T}$ is a certain horizontal distribution (see  \cite{Gindi1}). We then have 

\begin{prop} 
\label{propTANT2}
$T_{K}\mathcal{T}^{\#}(J)= V_{K}\mathcal{T}^{\#}(J) \oplus  H_{K}^{\nabla}\mathcal{T}.$
\end{prop}
Note that $V_{K}\mathcal{T}^{\#}(J)= T_{K}\mathcal{T}(E_{\pi(K)})^{\#}(J)$ was already described in Proposition \ref{PropTANTAU}.
\begin{proof}[Proof of Proposition \ref{propTANT2}]
The proof is similar to that of Proposition \ref{PropTCSPLIT}.
\end{proof}
\section{General Theorems about the $M^{\delta}$}
\label{SecGTM}
\subsection{A Twistor Point of View}
\label{SecATVP}
Let us return to the general setup of Section \ref{SecHSMAIN} and consider the respective parallel and holomorphic sections $J$ and $K$ of $(\mathcal{C}(E),\mathcal{J}^{(\nabla,I)}) \longrightarrow (M,I)$.  We are now prepared to carry out the set of ideas that were laid out at the end of that section---to  realize the $M^{\delta}:=\{M_{\leq s}, M_{(\leq r, \pm)}, M^{\#} \}$ as the intersection of certain complex submanifolds and subvarieties in $\mathcal{C}$ and to use this point of view to derive a number of corollaries about the $M^{\delta}$.  
 
 To begin, we will realize the $M^{\#}$ as the intersection of complex submanifolds in $\mathcal{C}$: First note that since $\nabla J=0$, by the previous section the $\mathcal{C}^{\#}(J)$ are complex submanifolds of $\mathcal{C}$. Secondly, since 
 $K:M \longrightarrow \mathcal{C}$ is holomorphic, we can use it to holomorphically embed $M$ into $\mathcal{C}$. We then have: 
 \begin{equation*}
K(M^{\#})= K(M) \cap \mathcal{C}^{\#}(J)
\end{equation*}
or alternatively
\begin{equation}
\label{EqINI}
M^{\#}= K^{-1}(\mathcal{C}^{\#}(J)). 
\end{equation}

To realize the other $M^{\delta}$ inside of $\mathcal{C}$, note that in the above equations we can respectively replace $M^{\#}$ with $M_{\leq s}$ or $M_{(\leq r, \pm)}$ and $\mathcal{C}^{\#}(J)$ with $\mathcal{C}_{\leq s}(J)$ or $\mathcal{C}_{(\leq r, \pm)}(J)$.

If we now equip $E$ with a fiberwise metric $g$ and choose $\nabla$ to be a metric connection and $J,K \in \Gamma{(\mathcal{T})}$ then it is clear that Equation \ref{EqINI} and its surrounding discussion would still be true if we were to replace $\mathcal{C}$ with $\mathcal{T}$. This then allows us to  view the $M^{\delta}$ as the intersection of complex submanifolds and subvarieties inside the twistor space $\mathcal{T}$. 

By using the properties of the subvarieties $\mathcal{C}^{\delta}$ and $\mathcal{T}^{\delta}$ which were derived in the previous section, we will now demonstrate two of our main corollaries of  this twistorial point of view of the $M^{\delta}$.

\subsection{Bounds on the $M^{\delta}$}
\label{SecBINT}
As a first corollary, we will bound the dimensions of the $M^{\delta}$. We will find that the bounds depend on whether $J$ and $K$ are sections of $\mathcal{C}$ or of $ \mathcal{T}$.

\subsubsection{$\mathcal{C}-$case}
Let us consider the setup of Theorem \ref{ThmPHS} where $J$ and $K$ are respectively parallel and holomorphic sections  of $(\mathcal{C}(E),\mathcal{J}^{(\nabla,I)}) \longrightarrow (M,I)$. Since $K:M \longrightarrow \mathcal{C}$ is a holomorphic map and, by the above discussion, $M_{(\leq r, \pm)}= K^{-1}(\mathcal{C}_{(\leq r, \pm)}(J))$, it follows from general theory that if $M_{(\leq r, \pm)}$ is nonempty in $M$ then the $codimM_{(\leq r, \pm)} \leq codim \mathcal{C}_{(\leq r, \pm)}(J)$. As we have already determined the codimensions of the $\mathcal{C}_{(\leq r, \pm)}(J)$ in Proposition \ref{PropCODCE}, we have

\begin{prop} 
\label{PropCBOND}
Let $dim_{\mathbb{C}}M=m$. If $M_{(\leq r, \pm)}$ is nonempty then the complex dimension of each of its components is $\geq m- (n-r)^{2}$.
\end{prop}
\begin{rmk}
The above proposition can also be proved in another way by applying the following lemma (see \cite{Fulton1}) to the first two holomorphic bundle maps given in Proposition \ref{PropBUNM}.

\begin{lemma} Let $E$ and $F$ be two rank $n$ holomorphic vector bundles fibered over a complex manifold $N$ and let $A:E \longrightarrow F$ be a holomorphic bundle map. If $N_{\leq k}:=\{x \in N | \ RankA|_{x} \leq k\}$ is nonempty then its complex codimension is $\leq (n-k)^{2}$. 
\end{lemma}
\end{rmk}

Either by using an analysis that is similar to the one first used to derive Proposition \ref{PropCBOND} or by deriving it directly from that proposition, we have:

\begin{prop}
\label{PropCBONDD}
Let $dim_{\mathbb{C}}M=m$. If $M^{\#}$ is nonempty then the complex dimension of each of its components is bounded as follows:
\begin{align*}
1)& \ dimM^{(m_{1},*)} \geqslant m- m_{1}^{2}\\
2)& \ dimM^{(*,m_{-1})} \geqslant m- m_{-1}^{2}\\
3)& \ dimM^{(m_{1},m_{-1})} \geqslant m- m_{1}^{2}-m_{-1}^{2}.
\end{align*}
\end{prop}
\subsubsection{$\mathcal{T}$-case} \label{Sec7.2.2}
Let us now consider the setup of Section \ref{SecTHSMC} so that $J$ and $K$ are respectively parallel and holomorphic sections  of $(\mathcal{T}(E,g),\mathcal{J}^{(\nabla,I)}) \longrightarrow (M,I)$. We will first focus on the $M_{(\leq r, +)}$. By using Proposition \ref{PropM'P}, we may split it into two disjoint open  subsets:
\[M_{(\leq r,+)}= M_{even}^{(\geq n-r,*)} \cup M_{even}^{(\geq n-r+1,*)},\]
where 
 \[M_{even}^{(\geq m_{1},*)}:= \bigcup_{k\in \mathbb{Z}_{\geq 0}} M^{(m_{1}+2k,*)}. \]

To bound the dimensions of $M_{even}^{(\geq m_{1},*)}$, we will express this subvariety as $K^{-1}(\mathcal{T}_{even}^{(\geq m_{1},*)}(J)),$ where $\mathcal{T}_{even}^{(\geq m_{1},*)}(J)$ is analogously defined and is, by Proposition \ref{PropCODTAU}, a holomorphic subvariety of $\mathcal{T}$.  As in the previous section, it then follows  that the $codim M_{even}^{(\geq m_{1},*)} \leq codim\mathcal{T}_{even}^{(\geq m_{1},*)}(J)$ and since we have already determined the codimensions of the $\mathcal{T}_{even}^{(\geq m_{1},*)}(J)$ in Proposition \ref{PropCODTAU}, we obtain: 

\begin{thm} Let $dim_{\mathbb{C}}M=m$. If $M_{even}^{(\geq m_{1},*)}$ is nonempty then the complex dimension of each of its components is $\geq m- \frac{m_{1}(m_{1}-1)}{2}$.

An analogous statement is true for $M_{even}^{(*,\geq m_{-1})}.$
\end{thm}

Either by using a similar analysis or by deriving it directly from the above theorem, we obtain:
\begin{thm}
\label{ThmBOUNDS}
If $M^{\#}$ is nonempty then the complex dimension of each of its components is bounded as follows:
\begin{align*}
1)& \ dimM^{(m_{1},*)} \geqslant m- \frac{m_{1}(m_{1}-1)}{2}\\
2)& \ dimM^{(*,m_{-1})} \geqslant m- \frac{m_{-1}(m_{-1}-1)}{2}\\
3)& \ dimM^{(m_{1},m_{-1})} \geqslant m- \frac{m_{1}(m_{1}-1)+ m_{-1}(m_{-1}-1)}{2}.
\end{align*}
\end{thm}
We will now list some cases that we will focus on in Section \ref{SecBM}.

\begin{cor}
\label{CorBOUNDS}
The following are bounds on the complex dimensions of some of the $M^{\#}$:
\begin{align*}
1) & \ dimM^{(1,2)}, dimM^{(2,1)} \geqslant m-1\\
2) & \ dimM^{(2,2)} \geqslant m-2 \\
3) & \ dimM^{(2,3)}, dimM^{(3,2)} \geqslant m-4 \\
4) & \ dimM^{(*,2)}, dimM^{(2,*)} \geqslant m-1.
\end{align*}
\end{cor}
\begin{rmk} Note that the bounds given in 1) follow from those in 4) since if $M^{(1,2)}$ is nonempty then it is open in 
$M^{(*,2)}.$ 
\end{rmk}
\subsection{The $M^{\#}$ along Curves}
\label{SecMAC}
We will now apply the twistor point of view of Section \ref{SecATVP} to derive another corollary about the $M^{\#}$. Unlike the previous one, this corollary will not use the holomorphicity of  the twistor space $\mathcal{C}$, rather it will just use the descriptions of $T_{K}\mathcal{C}^{\#}(J)$ and $T_{K}\mathcal{T}^{\#}(J)$ that were given in Propositions \ref{PropTCSPLIT} and \ref{propTANT2}. To describe it, consider the setup of a rank $2n$ real  vector bundle $E \longrightarrow M$ equipped with a connection $\nabla$. Also let $J$ and $K$ be sections of $\pi:\mathcal{C}(E)\longrightarrow M$ such that $\nabla J=0$.  Note that we do not impose any conditions on $K$ or on the curvature of $\nabla$; nor do we require $M$ to be even dimensional.

Given $x$ in $M^{\#}$ and $v \in T_{x}M$, the goal that we are currently working on is to derive necessary and sufficient  conditions for there to exist a curve $\gamma$ in $M$  such that $\gamma'(0)=v$ and $\gamma(t) \in M^{\#}$ for at least small $t \in \mathbb{R}$. We are interested in these conditions because they can be used to derive both upper and lower bounds on the dimensions of the $M^{\#}$ as described below.

As a first step, we will now show how to use twistor spaces to give a natural geometrical derivation of certain necessary conditions on $\nabla_{v} K$. 

\begin{prop}
\label{PropMPC}
Let $x \in M^{\#}$ and suppose that there is a curve $\gamma$ in $M$ such that $\gamma(0)=x$, $\gamma(t)$ lies in  $M^{\#}$ for small $t \in \mathbb{R}$ and $\gamma'(0)=v$. Then
\begin{align*}
1) & \text{ for }  M^{\#}= M^{(m_{1},*)}, \nabla_{v} K: Ker(J+K) \longrightarrow Im(J+K)\\
2) & \text{ for }  M^{\#}= M^{(*,m_{-1})}, \nabla_{v} K: Ker(J-K) \longrightarrow Im(J-K)\\
3) & \text{ for }  M^{\#}= M^{(m_{1},m_{-1})},  \nabla_{v} K: Ker(J+K) \longrightarrow Im(J+K) \text{ and }\\ 
    &  \ \ \ \ \ \ \ \  \ \ \ \ \ \ \ \ \ \ \ \ \ \ \ \ \ \ \ \  \nabla_{v} K: Ker(J-K) \longrightarrow Im(J-K).
\end{align*}
\end{prop}

\begin{proof}
As $\gamma(t) \in M^{\#}$, $K(\gamma(t))\in \mathcal{C}^{\#}(J)$, so that 
$K_{*}v \in T_{K}\mathcal{C}^{\#}(J)$. Since by Proposition \ref{PropTCSPLIT} \[T_{K}\mathcal{C}^{\#}(J)= V_{K}\mathcal{C}^{\#}(J) \oplus H^{\nabla}_{K}\mathcal{C},\] the vertical projection $P^{\nabla}$ of $K_{*}v$ lies in $V_{K}\mathcal{C}^{\#}(J)=T_{K}\mathcal{C}(E_{x})^{\#}(J)$.  Now by \cite{Gindi1}, $P^{\nabla}= \pi^{*}\nabla \phi$, where $\phi \in \Gamma(\pi^{*}EndE)$ is defined by $\phi|_{j}=j$, and since 
\[\pi^{*}\nabla_{K_{*}v} \phi=  \pi^{*}\nabla_{K_{*}v} (\pi^{*}K)= \nabla_{v}K,\] 
we find that $\nabla_{v}K \in T_{K}\mathcal{C}(E_{x})^{\#}(J)$. The proof of the proposition then follows from Proposition \ref{PropTANC}.  
\end{proof}
\begin{rmk} 
\label{RmkSSNC}
Although Proposition \ref{PropMPC} can certainly be proved by more direct methods that do not involve twistor spaces, we are currently using the twistor point of view of the $M^{\#}$  to derive stronger results (at least in the case when there are certain differential conditions imposed on $K$).
\end{rmk}

Let us now further suppose that $E$ is equipped with a fiberwise metric $g$   and $J$ and $K$ are sections of $\pi:\mathcal{T}(E,g)\longrightarrow M$. In this case we have derived a result analogous to Proposition \ref{PropMPC}. To state it, let us first orthogonally split  $E= Im[J,K] \oplus ker(J+K) \oplus ker(J-K)$ at the point $x \in M$ and define $P_{0}$, $P_{1}$ and $P_{-1}$ to be the corresponding projection operators. We then have:
\begin{prop}
\label{PropMPC2}
Let $x \in M^{\#}$ and suppose that there is a curve $\gamma$ in $M$ such that $\gamma(0)=x$, $\gamma(t)$ lies in  $M^{\#}$ for small $t \in \mathbb{R}$ and $\gamma'(0)=v$. Then
\begin{align*}
1) & \text{ for }  M^{\#}= M^{(m_{1},*)}, P_{1}(\nabla_{v} K) P_{1}=0\\
2) & \text{ for }  M^{\#}= M^{(*,m_{-1})}, P_{-1}(\nabla_{v} K) P_{-1}=0\\
3) & \text{ for }  M^{\#}= M^{(m_{1},m_{-1})}, P_{1}(\nabla_{v} K) P_{1}=0 \text{ and } P_{-1}(\nabla_{v} K) P_{-1}=0.
\end{align*}
\end{prop}
\begin{proof}
The proof follows directly from Proposition \ref{PropMPC} and the fact that $Im(J+K)=Im[J,K] \oplus ker(J-K)$. Note that if we also assume that $\nabla g=0$ then one can alternatively derive the above proposition by replacing $\mathcal{C}$ in the proof of Proposition \ref{PropMPC} with $\mathcal{T}$ and by using the description of $T_{K}\mathcal{T}(E_{x})^{\#}(J)$ given in Proposition \ref{PropTANTAU}. 
\end{proof}
It follows that if there exists a $v \in T_{x}M$ such that $P_{1}(\nabla_{v} K) P_{1} \neq 0$ then $M$ cannot equal $M^{(m_{1},*)}$ along any curve $\gamma$ that satisfies $\gamma'(0)=v$, i.e., the dimension of  $ker(J+K)$ along any such $\gamma$ must always change. Hence Propositions \ref{PropMPC} and \ref{PropMPC2} can be used to derive upper bounds on the dimensions of the $M^{\#}$. We refer to our forthcoming papers for explicit examples. 

We will now consider the holomorphic twistor spaces of bihermitian manifolds and will apply Theorem \ref{ThmBOUNDS} to study certain Poisson structures on the manifold.

\section{Bihermitian Manifolds}
\label{SecBM}
\subsection{Subvarieties and Bundle Maps}
\label{SecBMS}
Let $(M,g,J_{+},J_{-})$ be a bihermitian manifold, as described in Section \ref{SecBG}, so that \[\nabla^{+}J_{+}=0 \text{  and  } \nabla^{-}J_{-}=0,\] where $\nabla^{\pm}=\nabla^{Levi} \pm \frac{1}{2}g^{-1}H$ for a closed three form $H$. To build holomorphic subvarieties inside of $(M,J_{+})$, first note that by Corollary \ref{CorCSBG}, the total space of 
$(\mathcal{T} (TM),\mathcal{J}^{(\nabla^{-},J_{+})}) \longrightarrow (M,J_{+})$ is a complex manifold. Also, since $\nabla^{-}J_{-}=0$, $J_{-}$ is a parallel section of $\mathcal{T}$ and, by Corollary \ref{CorSKTHI}, $J_{+}$ is a holomorphic section. (Note that the holomorphicity condition on $J_{+}$ is equivalent to $J_{+}\nabla^{-}J_{+}=\nabla^{-}_{J_{+}}J_{+}$, which in turn can be shown to be equivalent to the integrability condition on $J_{+}$.) As these results are also true if we were to interchange $+$ with $-$, by Theorem \ref{ThmPHS} we have
\begin{prop}
\label{PropSUBBM}
Let $(M,g,J_{+},J_{-})$ be a bihermitian manifold. The following are holomorphic subvarieties of $M$ with respect to both $J_{+}$ and $J_{-}$:
\begin{align*}
&1)\ M_{\leq s}=\{x\in M | \ Rank[J_{+},J_{-}]|_{x} \leq 2s \} \\ 
&2)\ M_{(\leq r, \pm)}=\{x\in M | \ Rank(J_{+} \pm J_{-})|_{x} \leq 2r \}.
\end{align*}
\end{prop}
Let us now fix the complex structure $J_{+}$ on $M$. As in the general case of Section \ref{SecTHS}, to derive the holomorphicity of the above subvarieties in $(M,J_{+})$, we will first consider the following holomorphic bundles:  $T^{1,0}_{-}$  and $T^{0,1}_{-}$, equipped with the $\overline{\partial}$-operator 
$\nabla^{-(0,1)}$, as well as $T^{1,0}_{+}$ and $T^{0,1}_{+}$, equipped with the $\overline{\partial}$-operator $(\nabla^{-})^{'(0,1)}:= \nabla^{-(0,1)} + \frac{1}{2}(\nabla^{-(0,1)}J_{+}) J_{+}$. By Section \ref{ExTF2}, this latter (0,1) connection equals $\nabla^{Ch(0,1)}$, where $\nabla^{Ch}$ is the Chern connection on $TM$ that is associated to $(g,J_{+})$. (Note that $\nabla^{Ch(0,1)}= \nabla^{-(0,1)}$ on $T^{0,1}_{+}$.)

By Proposition \ref{PropBUNM}, we then have: 

\begin{prop}
\label{PropBUNBG}
Let $(M,g,J_{+},J_{-})$ be a bihermitian manifold. The following are holomorphic maps between the specified bundles  that are fibered over $(M,J_{+})$:
\begin{align*}
&1)\ J_{+}+J_{-}: T^{0,1}_{+} \longrightarrow T^{0,1}_{-}  \ \ \ \ \ \ \  2)\ J_{+}-J_{-}: T^{0,1}_{+} \longrightarrow T^{1,0}_{-} \\
&3)\ J_{+}+J_{-}: T^{1,0}_{-} \longrightarrow T^{1,0}_{+} \ \ \ \ \ \ \ 4)\ J_{+}-J_{-}: T^{0,1}_{-} \longrightarrow T^{1,0}_{+} \\
& \ \ \ \ \  \ \  \ \ \ \ \  \ \  \ \ 5)\ [J_{+},J_{-}]: T^{0,1}_{+} \longrightarrow T^{1,0}_{+}. 
\end{align*}
With the appropriate holomorphic structures on the bundles, the above statement is also true if we were to interchange $+$ with $-$. 
\end{prop}

\begin{cor}
\label{CorHBM}
$[J_{+},J_{-}]g^{-1} \in \Gamma(\wedge^{2}TM)$ induces a holomorphic section of $\wedge^{2}T^{1,0}_{+} \longrightarrow (M,J_{+})$ and of $\wedge^{2}T^{1,0}_{-} \longrightarrow (M,J_{-})$. 
\end{cor}

\begin{rmk}
Note that Corollary \ref{CorHBM} was first derived in \cite{Hitchin2} by using other methods and thus the holomorphicity of $M_{\leq s}$ is already known in the literature. However, the holomorphicity of the other maps in Proposition \ref{PropBUNBG} and  the holomorphicity of $M_{(\leq r, \pm)}$ are new to the literature. \footnote{ We note here that Marco Gualtieri has derived the holomorphicity of the $M_{(\leq r, \pm)}$ by using the generalized geometry description of bihermitian geometry.}
\end{rmk}
\begin{rmk}
\label{RmkIND2}
By Proposition \ref{PropBUNM2}, the holomorphicity of $[J_{+},J_{-}]$, as given in Part 5) of Proposition \ref{PropBUNBG}, is independent of the bihermitian condition $dH=0$. 
\end{rmk}

\subsection{A Twistor Point of View and Holomorphic Poisson Structures}
Using the holomorphic maps of Proposition \ref{PropBUNBG}, we not only obtain the holomorphic subvarieties $M_{\leq s}$ and $M_{(\leq r,\pm)}$ but also the subvarieties $M^{(m_{1},m_{-1})}$ and the rest of the $M^{\delta}$ which were defined for a more general setup in Definition \ref{DefMPOUND} (see also Notation \ref{notaMPOUND22}). There is an added interest in studying these subvarieties in the bihermitian setup because they are related to known real and holomorphic Poisson structures \cite{Lyak1,Hitchin2} (see also \cite{Apost1}): 
\begin{prop}
\label{PropSTRUC}
 Let $(M,g, J_{+}, J_{-})$ be a bihermitian manifold. The following are real Poisson structures on $M$:
\begin{align*}
   \sigma=[J_{+},J_{-}]g^{-1} \ \ \ \text{ and } \ \ \ \lambda_{\pm}= (J_{+} \pm J_{-})g^{-1}.
\end{align*}
Moreover, if we let $\sigma_{+}$ be the (2,0) component of $\sigma$ with respect to $J_{+}$ then it is a holomorphic Poisson structure on $(M,J_{+}).$ The same statement is true if we replace $+$ with $-.$
\end{prop}

   It then follows that some of the $M^{\delta}$ are the degeneracy and constant rank loci of the above Poisson structures while others refine the structure of these loci. More specifically, by Proposition \ref{PropMS} we can use the $M^{(m_{1},m_{-1})}$ to decompose the constant rank loci of the holomorphic Poisson structure $\sigma$ into open subsets:

\begin{equation*}
M_{s}= \bigcup_{m_{1}+m_{-1}=m-s} M^{(m_{1},m_{-1})},
\end{equation*}
where $dim_{\mathbb{C}}M=m.$
 
Our goal then is to determine which $M^{(m_{1},m_{-1})}$ exist in a given $M_{s}$ and, for those that do, derive bounds on their dimensions. This would allow us to understand how the dimensions of $ker(J_{+}+ J_{-})$ and $ker(J_{+}-J_{-})$ vary along the constant rank loci of $[J_{+},J_{-}]$, which are the loci of the holomorphic Poisson structure $\sigma$. 

To derive such results, we will first view, using Section \ref{SecATVP}, the Poisson loci and the $M^{(m_{1},m_{-1})}$ as intersections of holomorphic subvarieties and complex submanifolds \textit{in the holomorphic twistor spaces} $(\mathcal{T} (TM),\mathcal{J}^{(\nabla^{-},J_{+})})$ and $(\mathcal{T} (TM),\mathcal{J}^{(\nabla^{+},J_{-})})$. For instance, using the holomorphic section $J_{+}: (M,J_{+}) \longrightarrow (\mathcal{T} (TM), \\ \mathcal{J}^{(\nabla^{-},J_{+})})$, we can express 
\begin{equation*}
J_{+}(M_{s})= J_{+}(M) \cap \mathcal{\mathcal{T}}_{s}(J_{-}),
\end{equation*}
where the $\mathcal{\mathcal{T}}_{s}(J_{-})$ are the holomorphic subvarieties in $\mathcal{T}$ that were defined in Proposition \ref{PropTESUB} for a more general setting.
Similarly, using the holomorphic section $J_{-}: (M,J_{-}) \longrightarrow (\mathcal{T} (TM),\mathcal{J}^{(\nabla^{+},J_{-})})$, we can express
\begin{equation*}
J_{-}(M_{s})= J_{-}(M) \cap \mathcal{\mathcal{T}}_{s}(J_{+}).
\end{equation*}
Moreover, as we know from Section \ref{SecATVP}, these equations are still true if we respectively replace $M_{s}$ with $M^{(m_{1},m_{-1})}$ or $M^{\delta}$ and ${\mathcal{T}}_{s}$ with ${\mathcal{T}}^{(m_{1},m_{-1})}$ or ${\mathcal{T}}^{\delta}$. 

Now it was precisely this twistor point of view that led us to derive bounds on the $M^{\delta}$ in Section \ref{Sec7.2.2} for a more general setup (see Theorem \ref{ThmBOUNDS}). In the bihermitian setting, this yields in particular the following bounds on the $M^{(m_{1},m_{-1})}$.

\begin{thm}
\label{ThmBOUNDS22} Let $(M,g, J_{+}, J_{-})$ be a bihermitian manifold of complex dimension $m$.
If $M^{\#}$ is nonempty then the complex dimension of each of its components is bounded as follows:
\begin{align*}
1)& \ dimM^{(m_{1},*)} \geqslant m- \frac{m_{1}(m_{1}-1)}{2}\\
2)& \ dimM^{(*,m_{-1})} \geqslant m- \frac{m_{-1}(m_{-1}-1)}{2}\\
3)& \ dimM^{(m_{1},m_{-1})} \geqslant m- \frac{m_{1}(m_{1}-1)+ m_{-1}(m_{-1}-1)}{2}.
\end{align*}
\end{thm}

\begin{cor}
\label{CorBOUNDS22}
The following are bounds on the complex dimensions of some of the $M^{\#}$:
\begin{align*}
1) & \ dimM^{(1,2)}, dimM^{(2,1)} \geqslant m-1\\
2) & \ dimM^{(2,2)} \geqslant m-2 \\
3) & \ dimM^{(2,3)}, dimM^{(3,2)} \geqslant m-4 \\
4) & \ dimM^{(*,2)}, dimM^{(2,*)} \geqslant m-1.
\end{align*}
\end{cor}
\begin{rmk} Note that the bounds given in 1) follow from those in 4) since if $M^{(1,2)}$ is nonempty then it is open in 
$M^{(*,2)}.$ 
\end{rmk}

We have thus used holomorphic twistor spaces to bound the dimensions of the $M^{(m_{1},m_{-1})}$, which are open subsets of the constant rank loci of the holomorphic Poisson structure $\sigma=[J_{+},J_{-}]g^{-1}$. In Section \ref{secBMBPG}, we will use these bounds to derive results about the \textit{existence} of the $M^{(m_{1},m_{-1})}$ in $M$, especially for the case when $M= \mathbb{CP}^{3}$.

\subsubsection{Bounds and Poisson Geometry}
\label{SecBPGB}
Before we present those existence results, we note here that one can use Hamiltonian flows associated with the Poisson structures of Proposition \ref{PropSTRUC} together with Proposition \ref{PropMS} to derive the following bounds on the dimensions of the $M^{\#}$. 
\begin{prop}
\label{PropBDBG}
Let $dim_{\mathbb{C}}M=m.$ If $M^{\#}$ is nonempty then the complex dimension of each of its components is bounded as follows:
\begin{align*}
1) & \dim M^{(m_{1},*)} \geqslant m-m_{1}\\
2) & \dim M^{(*,m_{-1})} \geqslant m-m_{-1}\\
3) & \dim M^{(m_{1},m_{-1})} \geqslant m-(m_{1}+m_{-1}).
\end{align*}
\end{prop}
\begin{rmk}
\label{RmkMODB}
Note that the bounds in 1) and 2) modify some of those in 3). For example, if $M^{(m_{1},1)}$ is nonempty then its complex dimension is really $\geqslant m-m_{1}$ and not just $m-m_{1}-1$. The reason is that by Proposition \ref{PropM'P}, $M^{(m_{1},1)}$ is open in $M^{(m_{1},*)}$. Also note that if $M^{(1,*)}$ is nonempty then its complex dimension is $m$. 
\end{rmk}

To compare the above bounds to the ones derived from twistor space, note that only a few of the twistor bounds of Theorem \ref{ThmBOUNDS22}, which we list in Corollary \ref{CorBOUNDS22}, are stronger than the Poisson bounds given in Proposition \ref{PropBDBG}. Though, in the next section we will give a corollary of the twistor bounds that cannot be derived by using the Poisson bounds alone.  
\begin{rmk}
\label{RmkTPBO} It can be shown that the fact that $\sigma$ and $\lambda_{\pm}$ are Poisson structures, as stated in Proposition \ref{PropSTRUC}, does not depend on the bihermitian condition $dH=0$. Hence the bounds of Proposition \ref{PropBDBG}  are true regardless of this condition.  This is to be compared to the bounds derived from holomorphic twistor spaces (Theorem \ref{ThmBOUNDS22}), where the condition $dH=0$ was certainly used. 
\end{rmk}
 \subsection{Existence Results}
 \label{secBMBPG}
 We will now use the twistor bounds in Corollary \ref{CorBOUNDS22} to derive the following existence result about the $M^{\#}$  in a bihermitian manifold (note $M_{0}:= M_{\leq 0}$). 

\begin{thm}
\label{thmER1}
Let $(M,g,J_{+},J_{-})$ be a bihermitian manifold such that $dim_{\mathbb{C}}M=3$ and $dim_{\mathbb{C}}M_{0}\leq 1$. Then $M^{(2,1)}$  and $M^{(1,2)}$ are empty in $M$.
\end{thm}
\begin{proof}
If $M^{(2,1)}$ were nonempty then by the twistor bounds of Corollary \ref{CorBOUNDS22}, $dim_{\mathbb{C}}M^{(2,1)} \geq 2$. Yet, $M^{(2,1)} \subset M_{0}$ and $dim_{\mathbb{C}}M_{0}\leq 1$. Hence $M^{(2,1)}$ is empty in $M$. A similar argument shows that $M^{(1,2)}$ is also empty.  
\end{proof}

We will now apply this theorem to further study the algebraic interaction of $J_{+}$ and $J_{-}$ on $M_{0}$, which is  the zero rank locus of the holomorphic Poisson structure $\sigma$ given in Proposition \ref{PropSTRUC}.
\begin{thm}
\label{thmER2}
Let $(M,g,J_{+},J_{-})$ be a connected bihermitian manifold such that $dim_{\mathbb{C}}M=3$ and $dim_{\mathbb{C}}M_{0}\leq 1$. Then $M_{0}= M^{(0,3)}$ or $M_{0}= M^{(3,0)}$.
\end{thm}

\begin{proof}
Since $M$ is connected and $dim_{\mathbb{C}}M=3$, by Propositions \ref{PropMS} and \ref{PropM'P}, $M_{0}= M^{(0,3)} \cup M^{(2,1)}$ or $M_{0}= M^{(3,0)} \cup M^{(1,2)}.$ By Theorem \ref{thmER1}, $M^{(2,1)}$  and $M^{(1,2)}$ are empty in $M$. Hence $M_{0}= M^{(0,3)}$ or $M_{0}= M^{(3,0)}$.

\end{proof}
Under the assumptions in the above theorem,  $J_{+}$ must then equal $J_{-}$ or $-J_{-}$ on the zero rank locus of the holomorphic Poisson structure $\sigma$, and $M^{(2,1)}$  and $M^{(1,2)}$ cannot exist in the manifold. 

As for some examples of the above setup, we have:

\begin{prop}
\label{propER3}
There exist bihermitian structures on $\mathbb{CP}^{3}$ that satisfy the conditions in Theorems \ref{thmER1} and \ref{thmER2}.
\end{prop}
\begin{proof}
Given the standard complex structure $I$ on $\mathbb{CP}^{3}$, there exists a holomorphic Poisson structure, $\tilde{\sigma}$, on $\mathbb{CP}^{3}$ that vanishes only on points and complex curves \cite{Polish1}. Using a construction from \cite{Gualt3} one may build bihermitian structures $(g,J_{+},J_{-})$ on $\mathbb{CP}^{3}$ such that $J_{+}=I$ and the constant rank loci of $[J_{+},J_{-}]$ are the same as those for $Re\tilde{\sigma}$.
\end{proof}

\begin{rmk}
Note that the Poisson bounds of Proposition \ref{PropBDBG} are too weak to be used to derive Theorems \ref{thmER1} and \ref{thmER2}. For they would only yield a lower bound of one on the complex dimensions of $M^{(2,1)}$ and $M^{(1,2)}$ (see Remark \ref{RmkMODB}). Thus we really needed the twistor bounds of Theorem \ref{ThmBOUNDS22} to arrive at our results. 
\end{rmk}

\begin{rmk}
We can use the other twistor bounds of Theorem \ref{ThmBOUNDS22} to derive results similar to Theorems \ref{thmER1} and \ref{thmER2} but in higher dimensions. 
\end{rmk}

\section{Acknowledgments}
I would like to thank Blaine Lawson, Nigel Hitchin and Jason Starr for helpful discussions.

\textsc{Department of Mathematics, UC Riverside, Riverside, CA 92521} \\

\textit{E-mail Address:} \texttt{Gindis@ucr.edu}

\end{large}

\begin{thebibliography}{9}
\bibitem{Apost1}
V. Apostolov, P. Gauduchon, G. Grantcharov, \textit{Bihermitian structures on complex surfaces,} Proc. London Math. Soc. \textbf{79} (1999), 414-428. Corrigendum: \textbf{92} (2006), 200-202.
\bibitem{Bis1}
J.M. Bismut, \textit{A local index theorem for non Kahler manifolds,} Math. Ann., \textbf{284} (1989) 681-699. 
\bibitem{Skt2}
G.Cavalcanti, \textit{SKT geometry}, \textbf{arXiv:1203.0493}.
\bibitem{Skt1}
A. Fino, M. Parton, and S. Salamon, \textit{Families of strong KT structures in six dimensions}, Comment. Math. Helv. \textbf{79} (2004), no. 2, 317Ð340.
\bibitem{Fulton1}
W. Fulton, \textit{Intersection Theory}, Springer-Verlag, Berlin-Heidelberg, 1998.
\bibitem{Rocek1}
S. J. Gates, C. M. Hull and M. Ro\v{c}ek, \textit{Twisted multiplets and new supersymmetric 
nonlinear $\sigma$ models,} Nuclear Phys. B \textbf{248} (1984), 157-186.
\bibitem{Guad1}
P. Gauduchon, \textit{Hermitian connections and Dirac operators,} Boll. Un. Mat. Ital. B(7) \textbf{11} (1997), no. 2, suppl., 257Ð288.
\bibitem{Gindi1}
S. Gindi, \textit{Integrable complex structures on twistor spaces}, \textbf{arXiv:1212.4138v2. }
\bibitem{Gindi3}
S. Gindi, \textit{Representation theory of the algebra generated by a pair of complex structures}, \textbf{arXiv:0804.3621}. 
\bibitem{Gualt1}
M. Gualtieri, \textit{Generalized complex geometry}, \textbf{math.DG/0401221}.
\bibitem{Gualt3}
M. Gualtieri, \textit{Branes on poisson varieties,} \textbf{arXiv:0710.2719v2}. 
\bibitem{Hitchin1}
N. Hitchin, \textit{Generalized Calabi-Yau manifolds,} Quart. J. Math. Oxford Ser.\textbf{54} (2003), 281-308. 
\bibitem{Hitchin2}
N. Hitchin, \textit{Instantons, Poisson structures and generalized Kahler geometry,} J Commun. Math. Phys. \textbf{265} (2006), 131-164.
\bibitem{Lyak1}
S. Lyakhovich, M. Zabzine, \textit{Poisson geometry of sigma models with extended supersymmetry,} Phys. Lett. B \textbf{548} (2002) 243-251.
\bibitem{Polish1}
A. Polishchuk, \textit{Algebraic geometry of Poisson brackets,} J. Math. Sci. \textbf{84} 
(1997), 1413-1444. 
\end{thebibliography}
\end{document}